\documentclass[11pt]{article}
\usepackage{latexsym,amsfonts,amssymb,amsmath,amsthm}
\usepackage{graphicx}

\usepackage[usenames,dvipsnames]{color}
\usepackage{ulem}

\begin{document}
\setlength{\baselineskip}{16pt}

\parindent 0.5cm
\evensidemargin 0cm \oddsidemargin 0cm \topmargin 0cm \textheight
22cm \textwidth 16cm \footskip 2cm \headsep 0cm

\newtheorem{theorem}{Theorem}[section]
\newtheorem{lemma}{Lemma}[section]
\newtheorem{proposition}{Proposition}[section]
\newtheorem{definition}{Definition}[section]
\newtheorem{example}{Example}[section]
\newtheorem{corollary}{Corollary}[section]

\newtheorem{remark}{Remark}[section]

\numberwithin{equation}{section}

\def\p{\partial}
\def\I{\textit}
\def\R{\mathbb R}
\def\C{\mathbb C}
\def\u{\underline}
\def\l{\lambda}
\def\a{\alpha}
\def\O{\Omega}
\def\e{\epsilon}
\def\ls{\lambda^*}
\def\D{\displaystyle}
\def\wyx{ \frac{w(y,t)}{w(x,t)}}
\def\imp{\Rightarrow}
\def\tE{\tilde E}
\def\tX{\tilde X}
\def\tH{\tilde H}
\def\tu{\tilde u}
\def\d{\mathcal D}
\def\aa{\mathcal A}
\def\DH{\mathcal D(\tH)}
\def\bE{\bar E}
\def\bH{\bar H}
\def\M{\mathcal M}
\renewcommand{\labelenumi}{(\arabic{enumi})}

\def\disp{\displaystyle}
\def\undertex#1{$\underline{\hbox{#1}}$}
\def\card{\mathop{\hbox{card}}}
\def\sgn{\mathop{\hbox{sgn}}}
\def\exp{\mathop{\hbox{exp}}}
\def\OFP{(\Omega,{\cal F},\PP)}
\newcommand\JM{Mierczy\'nski}
\newcommand\RR{\ensuremath{\mathbb{R}}}
\newcommand\CC{\ensuremath{\mathbb{C}}}
\newcommand\QQ{\ensuremath{\mathbb{Q}}}
\newcommand\ZZ{\ensuremath{\mathbb{Z}}}
\newcommand\NN{\ensuremath{\mathbb{N}}}
\newcommand\PP{\ensuremath{\mathbb{P}}}
\newcommand\abs[1]{\ensuremath{\lvert#1\rvert}}

\newcommand\normf[1]{\ensuremath{\lVert#1\rVert_{f}}}
\newcommand\normfRb[1]{\ensuremath{\lVert#1\rVert_{f,R_b}}}
\newcommand\normfRbone[1]{\ensuremath{\lVert#1\rVert_{f, R_{b_1}}}}
\newcommand\normfRbtwo[1]{\ensuremath{\lVert#1\rVert_{f,R_{b_2}}}}
\newcommand\normtwo[1]{\ensuremath{\lVert#1\rVert_{2}}}
\newcommand\norminfty[1]{\ensuremath{\lVert#1\rVert_{\infty}}}

\title{Diffusive KPP Equations with Free Boundaries in Time Almost Periodic Environments: I. Spreading and Vanishing Dichotomy}

\author{Fang Li\\
School of Mathematical Sciences\\
University of Science and Technology of China\\
Hefei, Anhui, 230026, P.R. China\\
and\\
Department of Mathematics and Statistics\\
Auburn University\\
Auburn University, AL 36849 \\
 \\
 Xing Liang\\
 School of Mathematical Sciences\\
  University of Science and Technology of China\\
  Hefei, Anhui, 230026, P.R. China\\
   \\
  and
  \\
   Wenxian Shen\\
   Department of Mathematics and Statistics\\
Auburn University\\
Auburn University, AL 36849\\
 }

\date{}
\maketitle

\noindent \textbf{Abstract.} In this series of papers, we investigate  the spreading and
vanishing dynamics of time almost periodic diffusive KPP equations with  free boundaries.
Such equations are used to characterize the spreading of a new
 species in time almost periodic environments with free boundaries representing the spreading fronts.
In this first part, we show that a spreading-vanishing
dichotomy occurs for such free boundary problems, that is,  the species either successfully spreads to all the new
environment and stabilizes at a time almost periodic positive solution, or it fails to establish and dies out eventually.
The results of this part extend the existing results on spreading-vanishing dichotomy for time and space independent, or time periodic and space independent,
or  time independent and space periodic
diffusive KPP equations with free boundaries. The extension is nontrivial and is ever done for the first time.

\medskip


\medskip



\section{Introduction}

This is the first part of a series of papers on  the spreading and
vanishing dynamics of  diffusive equations with free boundaries of the form,
\begin{equation}
\label{main-eq}
\begin{cases}
u_t=u_{xx}+uf(t,x,u),\quad &t>0,\,\,  0<x<h(t)\cr h^{'}(t)=-\mu u_x(t,h(t)),\quad &t>0,\cr
u_x(t,0)=u(t,h(t))=0, \quad &t>0,\cr h(0)=h_0, u(0,x)=u_0(x),\quad &0\leq x\leq h_0,
\end{cases}
\end{equation}
where $x=h(t;u_0,h_0)$ is the moving boundary to be determined, $\mu$, $h_0$ are given
positive constants, and the initial function $u_0(x)$ satisfies
 \begin{equation}
 \label{initial-value}
 u_0\in C^2([0,h_0]),  \  u^{'}_0(0)=u_0(h_0)=0, \ {\rm and} \ u_0>0 \ {\rm in} \ [0,h_0).
 \end{equation}

We assume that
 $f(t,x,u)$ is a $C^1$ function in $t\in\RR$, $x\in\RR$, and
$u\in\RR$; $f(t,x,u)<0$ for $u\gg 1$; $f_u(t,x,u)<0$ for $u\ge 0$;
and $f(t,x,u)$ is almost periodic in $t$ uniformly with respect to
$x\in\RR$ and $u\in\RR$ (see (H1), (H2) in subsection 2.1 for detail). Here is a typical example
of such functions,
$f(t,x,u)=a(t,x)-b(t,x)u$, where $a(t,x)$ and $b(t,x)$ are almost periodic in $t$ and periodic in $x\in\RR$, and
$\inf_{t\in\RR,x\in\RR}b(t,x)>0$.

 Observe  that for any given $u_0$ satisfying \eqref{initial-value},
 \eqref{main-eq} has a  unique (local) solution  $(u(t,x;u_0,h_0)$, $h(t;u_0,h_0))$ with
  $u(0,x;u_0,h_0)=u_0(x)$ and $h(0;u_0,h_0)=h_0$ (see \cite{DuGuPe}). Moreover,
by comparison principle for parabolic equations, $(u(t,x;u_0,h_0),h(t;u_0,h_0))$ exists for all
$t>0$ and $u_x(t,h(t))<0$. Hence $h(t;u_0,h_0)$ increases as $t$ increases.

Equation \eqref{main-eq}  with $f(t,x,u)=u(a-bu)$ and $a$ and $b$ being
two positive constants was introduced by Du and Lin in \cite{DuLi}
to understand the spreading of species. A great deal of previous mathematical investigation on the spreading of species (in one space dimension case)
has been based on diffusive equations of the form
\begin{equation}
\label{kpp-general-eq}
u_t=u_{xx}+u f(t,x,u),\quad x\in\RR,
\end{equation}
where $f(t,x,u)<0$ for $u\gg 1$ and $f_u(t,x,u)<0$ for $u\ge 0$. Thanks to the pioneering works of Fisher (\cite{Fis}) and Kolmogorov, Petrowsky, Piscunov
(\cite{KPP}) on the following special case of \eqref{kpp-general-eq}
\begin{equation}
\label{kpp-special-eq}
u_t=u_{xx}+u(1-u),\quad x\in\RR,
\end{equation}
\eqref{main-eq}, resp. \eqref{kpp-general-eq}, is referred to as diffusive Fisher or KPP equation.

One of the central problems for both \eqref{main-eq} and \eqref{kpp-general-eq} is
to understand their spreading dynamics. For \eqref{kpp-general-eq}, this is closely related to
spreading speeds and transition fronts of \eqref{kpp-general-eq} and has been widely studied.
When $f(t,x,u)$ is independent of $t$ and $x$ or is periodic in $t$ and $x$, the spreading dynamics
for \eqref{kpp-general-eq} is quite well understood. For example, assume  that $f(t,x,u)$ is periodic in
$t$ with period $T$ and periodic in $x$  with period $p$, and that
$u\equiv 0$ is a linearly unstable solution of \eqref{kpp-general-eq} with respect to periodic perturbations. Then it is known that
\eqref{kpp-general-eq} has a unique positive periodic solution $u^*(t,x)$ ($u^*(t+T,x)=u^*(t,x+p)=u^*(t,x)$)  which is asymptotically stable with
respect to periodic perturbations   and
 it
has been proved
 that   there is a positive constant $c^*$ such that for every $c\geq c^*$, there is a periodic
 traveling wave solution $u(t,x)$ connecting $u^*$ and $u\equiv 0$
 with speed $c$ (i.e. $u(t,x)=\phi(x-ct,t,x)$ for some $\phi(\cdot,\cdot,\cdot)$ satisfying that
 $\phi(\cdot,\cdot+T,\cdot)=\phi(\cdot,\cdot,\cdot+p)=\phi(\cdot,\cdot,\cdot)$ and $\phi(-\infty,\cdot,\cdot)=u^*(\cdot,\cdot)$ and
 $\phi(\infty,\cdot,\cdot)=0$), and there is no such traveling
 wave solution of slower speed (see \cite{LiZh, Nad, NoXi, Wei}).  Moreover, the minimal wave speed $c^*$ is
 of the following spreading  property and is hence called the spreading speed of \eqref{kpp-general-eq}:
 for any given $u_0\in C_{\rm unif}^b(\RR,\RR^+)$ with non-empty support,
 \begin{equation}
 \label{kpp-spreading-eq}
 \begin{cases}
 \lim_{|x| \le c^{'}t,t\to\infty}[u(t,x;u_0)-u^*(t,x)]=0\quad \forall\,\, c^{'}<c^*\cr
 \lim_{|x| \ge c^{''}t,t\to\infty} u(t,x;u_0)=0\quad \forall\,\, c^{''}> c^*,
 \end{cases}
 \end{equation}
 where $u(t,x;u_0)$ is the solution of \eqref{kpp-general-eq} with $u(0,x;u_0)=u_0(x)$
 (see \cite{LiZh, Wei}).

The spreading property \eqref{kpp-spreading-eq} for \eqref{kpp-general-eq} in the case that $f(t,x,u)$ is periodic in $t$ and $x$
 implies that spreading always happens for a solution of
\eqref{kpp-general-eq} with a positive initial function, no matter how small the positive initial function is.
The following strikingly different  spreading scenario has been proved for  \eqref{main-eq} in the case that
$f(t,x,u)\equiv f(u)$ (see \cite{DuGu, DuLi}):  it exhibits a spreading-vanishing dichotomy in the sense that for any given positive initial data
$u_0$ satisfying \eqref{initial-value} and $h_0$, either vanishing occurs (i.e. $\lim_{t\to\infty}h(t;u_0,h_0)<\infty$ and $\lim_{t\to\infty}u(t,x;u_0,h_0)=0$) or
spreading occurs (i.e. $\lim_{t\to\infty}h(t;u_0,h_0)=\infty$ and $\lim_{t\to\infty}u(t,x;u_0,h_0)=u^*$ locally uniformly in $x\in \RR^+$, where
$u^*$ is the unique positive solution of $f(u)=0$). The above spreading-vanishing dichotomy for \eqref{main-eq} with
$f(t,x,u)\equiv f(u)$ has also been extended to the cases that $f(t,x,u)$ is periodic in $t$ or that $f(t,x,u)$ is independent of $t$ and periodic
in $x$ (see \cite{DuGuPe,DuLia}). The spreading-vanishing dichotomy proved for \eqref{main-eq} in \cite{DuGu, DuGuPe, DuLi, DuLia}
 is well supported by some empirical evidences, for example, the introduction of several
bird species from Europe to North America in the 1900s was successful only after many initial
attempts (see \cite{LoHoMa, ShKa}).

In reality, many evolution systems in biology are subject to non-periodic time and/or space variations. It is therefore
of great importance to investigate the spreading dynamics for both \eqref{main-eq} and \eqref{kpp-general-eq} with general time and space
dependent $f(t,x,u)$.  The spreading dynamics for \eqref{kpp-general-eq} with non-periodic time and/or space dependence has been
studied by many people recently
(see \cite{BeHa, BeHaNa, BeNa, HuSh, KoSh, NaRo, NoXi1, She1, She2, She3, TaZhZl, Zla}, etc.). However, there is little study
on the spreading dynamics for \eqref{main-eq} with non-periodic time and space dependence.

The objective of the current series
of papers is to investigate the spreading-vanishing dynamics of \eqref{main-eq} in the case that $f(t,x,u)$ is almost periodic in $t$,
that is, to investigate  whether the population will
successfully establishes itself
in the entire space (i.e. spreading occurs), or  it fails to establish and vanishes
eventually (i.e. vanishing occurs).
Roughly speaking, for given $(u_0,h_0)$, if $h_\infty=\lim_{t\to\infty}h(t;u_0,h_0)=\infty$ and for any $M>0$,
$\liminf_{t\to\infty}\inf_{0\le x\le M}u(t,x;u_0,h_0)>0$,
we say {\it spreading} occurs. If $h_\infty<\infty$ and $\lim_{t\to\infty}u(t,x;u_0,h_0)=0$, we say {\it vanishing} occurs (see Definition
\ref{spreading-vanishing-def} for detail). We say a positive number $c^*$ is a {\it spreading speed} of \eqref{main-eq} if for any $(u_0,h_0)$ such that the spreading occurs,
$$\lim_{t\to\infty}\frac{h(t;u_0,h_0)}{t}=c^*$$
 and
 $$
 \liminf_{0\le x\le c^{'}t,t\to\infty}u(t,x;u_0,h_0)>0\quad \forall \,\, c^{'}<c^*
 $$
 (see Definition
\ref{spreading-vanishing-def} for detail).

In this first part of the series of the papers, we focus on the study of spreading and vanishing dichotomy
scenario for \eqref{main-eq}.  Among others,
we  prove the following spreading and vanishing dichotomy:

\medskip

\noindent $\bullet$  {\it Assume (H1)-(H5) stated in subsection 2.1. For any given $u_0$ satisfying
\eqref{initial-value}, either spreading occurs or vanishing occurs.
Moreover, there is $l^*>0$ such that for any given $u_0$ satisfying \eqref{initial-value}, vanishing occurs if and only if
$h_\infty\le l^*$} (see Theorem \ref{main-thm2}).

\medskip

To characterize the detailed  spreading and vanishing dynamics of \eqref{main-eq},
we also consider the following  fixed boundary problem on half line,
\begin{equation}
\label{aux-main-eq1}
\begin{cases}
u_t=u_{xx}+uf(t,x,u),\quad x\in (0,\infty)\cr
u_x(t,0)=0.
\end{cases}
\end{equation}

Observe that if $u^*(t,x)$ is a solution of \eqref{aux-main-eq1} and  $u_0(x)\le u^*(0,x)$ for $x\in [0,h_0]$, then
$u(t,x;u_0,h_0)\le u^*(t,x)$ for $0\le x\le h(t;u_0,h_0)$.  Among others, we prove that

\medskip

\noindent $\bullet$ {\it Assume (H1)-(H5) stated in subsection 2.1.
 \eqref{aux-main-eq1} has a unique time almost periodic positive solution $u^*(t,x)$}  (see Theorem \ref{main-thm1}) {\it and
for any given $u_0$ satisfying \eqref{initial-value}, if spreading occurs in \eqref{main-eq}, then $u(t,x;u_0,h_0)-u^*(t,x)\to 0$ as $t\to\infty$ locally uniformly in $x\ge 0$} (see
Theorem \ref{main-thm2}).

\medskip

We note that the techniques for \eqref{main-eq}  can be modified to study the
following double fronts free boundary problem:
\begin{equation}
\begin{cases}
\label{main-doub-eq}
u_t=u_{xx}+uf(t,x,u), \quad  &t>0,\,\,  g(t)<x<h(t), \cr
u(t,g(t))=0, g^{'}(t)=-\mu u_x(t,g(t)),  \quad &t>0,\cr
u(t,h(t))=0, h^{'}(t)=-\mu u_x(t,h(t)), \quad &t>0,\cr
u(0,x)=u_0(x),\quad & h_0\le x\le g_0 \cr
h(0)=h_0,\,\, g(0)=g_0
\end{cases}
\end{equation}
where both $x=g(t)$ and $x=h(t)$ are to be determined and $u_0$ satisfies
\begin{equation}
\label{initial-value-1}
\begin{cases}
u_0\in C^2([g_0,h_0]) \cr
u_0(g_0)=u_0(h_0)=0 \ \ and  \ \ u_0>0 \ \ in \ \ (g_0,h_0).
\end{cases}
\end{equation}
Under the assumptions (H1), (H2), $(H4)^*$, and (H5) (see section 6 for $(H4)^*$),  spreading-vanishing dichotomy for \eqref{main-doub-eq} also holds.
In particular, we prove that

\medskip
\noindent $\bullet$ {\it Assume (H1), (H2), $(H4)^*$, and (H5). For given $h_0>0$ and  $u_0$ satisfying \eqref{initial-value-1},
 either
 $h_\infty-g_\infty<\infty$ and
$\lim_{t\to+\infty} u(t,x;u_0,h_0,g_0)=0$ uniformly in $x$,
or   $h_{\infty}=-g_{\infty}=\infty$ and
$\liminf_{t\to\infty}\inf_{|x|\le M}u(t,x;u_0,h_0,g_0)>0$
for any $M>0$}
(see Proposition \ref{spreading-vanishing-doub-prop}).

\medskip

In the second part of the series of the papers, we will study the existence of  spreading speeds
for \eqref{main-eq} and the existence of time almost periodic semi-wave solutions of the
following free boundary problem associated to \eqref{main-eq},
\begin{equation}
\label{aux-main-eq3}
\begin{cases}
u_t=u_{xx}+uf(t,x,u),\quad &t>0, \,\, -\infty<x<h(t)\cr
u(t,h(t))=0,\quad &t>0,\cr
h^{'}(t)=-\mu u_x(t,h(t)), \quad &t>0.
\end{cases}
\end{equation}
 If $(u(t,x),h(t))$ is an entire solution of \eqref{aux-main-eq3},
it is called a {\it semi-wave solution} of \eqref{aux-main-eq3}.

 The rest of this paper is organized as follows.  In Section 2, we introduce the definitions and standing assumptions and
 state the main results of the paper. We present preliminary materials in Section 3 for the use in later sections.
 Section 4 is devoted to the investigation of time almost periodic KPP equation \eqref{aux-main-eq1}
  on the half line and to the proof of Theorem \ref{main-thm1}. In Section 5,
 we explore the spreading and vanishing dichotomy scenario of \eqref{main-eq} and prove
 Theorem \ref{main-thm2}.  The paper is ended with some remarks on spreading-vanishing dichotomy for \eqref{main-doub-eq} in Section 6.

\section{Definitions, Assumptions, and Main Results}

In this section, we introduce the  definitions and standing assumptions, and state the main results.

\subsection{Definitions and assumptions}

In this subsection, we introduce the definitions and standing assumptions. We first recall the definition of almost periodic
functions, next recall the definition of principal Lyapunov exponents for some linear parabolic equations, then state
the standing assumptions, and finally introduce the definition of spreading and vanishing for \eqref{main-eq}.

\begin{definition}[Almost periodic function]
\label{almost-periodic-def}
\begin{itemize}

\item[(1)] A continuous function $g:\RR\to \RR$ is called {\rm almost periodic} if
for any $\epsilon>0$, the set
$$
T(\epsilon)=\{\tau\in\RR\,|\, |f(t+\tau)-f(t)|<\epsilon\,\, \, \text{for all}\,\, t\in\RR\}
$$
is relatively dense in $\RR$.

\item[(2)] Let $g(t,x,u)$ be a continuous function of $(t,x,u)\in\RR\times\RR^m\times\RR^n$. $g$ is said to be {\rm almost periodic in $t$ uniformly with respect to $x\in\RR^m$ and
$u$ in bounded sets} if
$g$ is uniformly continuous in $t\in\RR$,  $x\in\RR^m$, and $u$ in bounded sets and for each $x\in\RR^m$ and $u\in\RR^n$, $g(t,x,u)$ is almost periodic in $t$.

\item[(3)] For a given almost periodic function $g(t,x,u)$, the hull $H(g)$ of $g$ is defined by
\begin{align*}
H(g)=\{\tilde g(\cdot,\cdot,\cdot)\,|\, & \exists t_n\to\infty \,\,\text{such that}\,\, g(t+t_n,x,u)\to \tilde g(t,x,u)\,\, \text{uniformly in}\,\, t\in\RR,\\
&
\,\, (x,u)\,\, \text{in bounded sets}\}.
\end{align*}
\end{itemize}
\end{definition}

\begin{remark}
\label{almost-periodic-rk}
(1) Let $g(t,x,u)$ be a continuous function of $(t,x,u)\in\RR\times\RR^m\times\RR^n$. $g$ is almost periodic in $t$ uniformly with respect to
$x\in\RR^m$ and $u$ in bounded sets  if and only if
 $g$ is uniformly continuous in $t\in\RR$,  $x\in\RR^m$, and $u$ in bounded sets and for any sequences $\{\alpha_n^{'}\}$,
$\{\beta_n^{'}\}\subset \RR$, there are subsequences $\{\alpha_n\}\subset\{\alpha_n^{'}\}$, $\{\beta_n\}\subset\{\beta_n^{'}\}$
such that
$$
\lim_{n\to\infty}\lim_{m\to\infty}g(t+\alpha_n+\beta_m,x,u)=\lim_{n\to\infty}g(t+\alpha_n+\beta_n,x,u)
$$
for each $(t,x,u)\in\RR\times\RR^m\times\RR^n$ (see \cite[Theorems 1.17 and 2.10]{Fink}).

(2) We may write $g(\cdot+t,\cdot,\cdot)$ as $g\cdot t(\cdot,\cdot,\cdot)$.
\end{remark}

For a given positive constant $l>0$ and a given $C^1$ function $a(t,x)$ which is almost periodic in $t$ uniformly in $x$ in bounded sets, consider
\begin{equation}
\label{linearized-eq1}
\begin{cases}
v_t=v_{xx}+a(t,x)v,\quad 0<x<l\cr v_x(t,0)=v(t,l)=0.
\end{cases}
\end{equation}

Let
$$
Y(l)=\{u\in C([0,l])\,|\, u(l)=0\}
$$
with the norm $\|u\|=\max_{x\in [0,l]}|u(x)|$ for $u\in Y(l)$.
Let $A=\Delta$ acting on $Y(l)$ with $\mathcal{D}(A)=\{u\in C^2([0,l])\cap Y(l)\,|\, u_x(0)=0\}$.
Note that  $A$ is a sectorial operator. Let $0<\alpha<1$ be such that $\mathcal{D}(A^\alpha)\subset C^1([0,l])$. Fix such an $\alpha$.
Let
\begin{equation}
\label{bounded-domain-space-eq1}
X(l)=\mathcal{D}(A^\alpha).
\end{equation}
Then $X(l)$ is strongly ordered Banach spaces with positive cone
$$
X^+(l)=\{u\in X(l)\,|\, u(x)\ge 0\}.
$$

Let
$$
X^{++}(l)={\rm Int}(X^+(l)).
$$
If no confusion occurs, we may write $X(l)$ as $X$.

By semigroup theory (see \cite{PA}), for any  $v_0\in X(l)$, \eqref{linearized-eq1}
has a unique solution $v(t,\cdot;v_0,a)$ with $v(0,\cdot;v_0,a)=v_0(\cdot)$.

For a given positive constant $l>0$ and a given $C^1$ function
 $a(t,x)$ which is almost periodic function in $t$ uniformly in $x$ in bounded sets, consider
also
\begin{equation}
\label{aaux-linearized-eq1}
\begin{cases}
v_t=v_{xx}+a(t,x)v,\quad 0<x<l\cr
v(t,0)=v(t,l)=0.
\end{cases}
\end{equation}
Let
$$
\tilde Y(l)=\{u\in C([0,l])\,|\, u(0)=u(l)=0\}.
$$
Let $A=\Delta$ acting on $\tilde Y(l)$ with $\mathcal{D}(A)=\{u\in C^2([0,l])\cap \tilde Y(l)\}$.
Note that  $A$ is a sectorial operator. Let $0<\alpha<1$ be such that $\mathcal{D}(A^\alpha)\subset C^1([0,l])$. Fix such an $\alpha$.
Let
\begin{equation}
\label{bounded-domain-space-eq2}
\tilde X(l)=\mathcal{D}(A^\alpha).
\end{equation}
Then, for any $v_0\in \tilde X(l)$, \eqref{aaux-linearized-eq1} has a unique solution
$\tilde v(t,\cdot;v_0,a)$ with $\tilde v(0,\cdot;v_0,a)=v_0(\cdot)$.

\begin{definition}[Principal Lyapunov exponent]
\label{principal-lyapunov-exp}
\begin{itemize}
\item[(1)] Let $V(t,a)v_0=v(t,\cdot;v_0,a)$ for $v_0\in X(l)$ and
$$
\lambda(a,l)=\limsup_{t\to\infty}\frac{\ln \|V(t,a)\|_{X(l)}}{t}.
$$
$\lambda(a,l)$ is called the {\rm principal Lyapunov exponent} of \eqref{linearized-eq1}.

\item[(2)] Let
$$
\tilde\lambda(a,l)=\limsup_{t\to\infty}\frac{\ln\|\tilde V(t,a)\|_{\tilde{X(l)}}}{t}
$$
where $\tilde V(t,a)v_0=\tilde v(t,\cdot;v_0,a)$ for $v_0\in\tilde X(l)$.
$\tilde\lambda(a,l)$ is called the {\rm principal Lyapunov exponent} of \eqref{aaux-linearized-eq1}.
\end{itemize}
\end{definition}

Let (H1)-(H5) be the following standing assumptions.

\medskip

\noindent {\bf (H1)} {\it $f(t,x,u)$ is $C^1$ in $(t,x,u)\in\RR^3$, $Df=(f_t,f_x,f_u)$ is bounded in $(t,x)\in\RR\times\RR$ and
in $u$ in bounded sets, and $f$ is monostable in $u$ in the sense that
there are $M>0$ such that $$\sup_{t\in\RR,x\in\RR,u\ge M}f(t,x,u)<0$$ and
$$\sup_{t\in\RR,x\in\RR,u\ge 0}f_u(t,x,u)<0.$$}

\medskip

\noindent{\bf (H2)} {\it $f(t,x,u)$ and $Df(t,x,u)=(f_t(t,x,u),f_x(t,x,u),f_u(t,x,u))$ are  almost periodic in $t$ uniformly with respect to $x\in\RR$ and
$u$ in bounded sets.
}

\medskip
\noindent {\bf (H3)} {\it There is $l^*>0$ such that
$\lambda(a(\cdot,\cdot),l)>0$ for $l>l^*$, where $a(t,x)=f(t,x,0)$.
}

\medskip

\noindent{\bf (H4)} {\it There are $y^*\ge 0$ and $L^*\ge 0$ such that
 $\tilde\lambda(a(\cdot,\cdot+y),l)>0$ for $y\ge y^*$ and $l\ge L^*$.}

 \medskip

\noindent {\bf (H5)}
 {\it For any given sequence $\{y_n^{'}\}\subset \RR$ and $\{g_{n}^{'}\}\subset H(f)$, there are  subsequences
$\{y_n\}\subset\{y_n^{'}\}$ and $\{g_n\}\subset\{g_n^{'}\}$ such that
$\lim_{n\to\infty} g_n(t,x+y_n,u)$ exists uniformly in $t\in\RR$ and $(x,u)$ in bounded sets.
}

\medskip

Assume (H1) and (H2).
We remark that, if $f(t,x,u)\equiv f(t,u)$, then  (H3) (resp. (H4))  holds if and only if
$\lim_{t\to\infty}\frac{1}{t}\int_0^t f(s,0)ds>0$ (see Lemma \ref{lyapunov-exponent-lm3} for the reasoning).
If $f(t,x,u)\equiv f(t,u)$ or $f(t,x,u)$ is periodic in $x$, (H5) is automatically true.

Consider \eqref{main-eq}. Throughout this paper, we assume (H1) and (H2).
For any given $u_0$ satisfying \eqref{initial-value},  \eqref{main-eq} has a unique solution
$(u(t,x;u_0,h_0),h(t;u_0,h_0))$ with $u(0,x;u_0,h_0)=u_0(x)$
and $h(0;u_0,h_0)=h_0$ (see \cite{DuGuPe}). By comparison principle for parabolic equations, $u(t,x;u_0,h_0)$ exists for all $t>0$ and
$u_x(t,x;u_0,h_0)\le 0$ for $t>0$.
Hence $h(t;u_0,h_0)$ is monotonically increasing, and therefore
there exists $h_{\infty}\in(0,+\infty]$ such that
$\lim_{t\to+\infty}h(t;u_0,h_0)=h_{\infty}$.

\begin{definition}[Spreading-vanishing and spreading speed]
\label{spreading-vanishing-def}
Consider \eqref{main-eq}.
\begin{itemize}
\item[(1)] For any given $u_0$ satisfying \eqref{initial-value}, let $h_\infty=\lim_{t\to\infty}h(t;u_0,h_0)$.
 It is said that the {\rm vanishing occurs} if $ h_\infty<\infty$ and
$\lim_{t\to\infty}\|u(t,\cdot;u_0,h_0)\|_{C([0,h(t)])}=0$. It is said that the {\rm spreading occurs}
 if $h_\infty=\infty$ and
 $\liminf_{t\to\infty} u(t,x;u_0,h_0)>0$ locally uniformly in
$x\in [0,\infty)$.

\item[(2)] A real number $c^*>0$ is called the {\rm spreading speed} of \eqref{main-eq} if for any $(u_0,h_0)$ such that
\eqref{initial-value} is  satisfied and the spreading occurs,
there holds
$$
\lim_{t\to\infty}\frac{h(t;u_0,h_0)}{t}=c^*
$$
and
$$
\liminf_{0\le x\le c^{'}t,t\to\infty}u(t,x;u_0,h_0)>0,\quad \forall\,\, c^{'}<c^*.
$$
\end{itemize}
\end{definition}

Biologically, spreading means that the free boundary  $x=h(t;u_0,h_0)$ goes to infinity
as $t\to\infty$ (i.e., $h_{\infty}=\infty)$, and population $u(t,x;u_0,h_0)$
successfully establishes itself
in the entire space. On the other
hand, vanishing means that the free boundary fails to move eventually, and
the population fails to establish and vanishes
eventually.

\subsection{Main results}

In this subsection, we state the main results of this paper. The first theorem is about the existence of time almost
periodic positive solution of \eqref{aux-main-eq1}.

\begin{theorem}[Almost periodic solutions]
\label{main-thm1}
 Consider \eqref{aux-main-eq1} and assume (H1)-(H5). Then there is a unique time almost periodic positive solution
$u^*(t,x)$ of \eqref{aux-main-eq1} and for any $u_0\in C_{\rm unif}^b([0,\infty),\RR^+)$ with $\inf_{x\in [0,\infty)}u_0(x)>0$,
$$
\lim_{t\to\infty}\|u(t,\cdot;u_0)-u^*(t,\cdot)\|_{C([0,\infty))}=0,
$$
where $u(t,x;u_0)$ is the solution of \eqref{aux-main-eq1} with $u(0,x;u_0)=u_0(x)$. If, in addition, $f(t,x,u)\equiv f(t,u)$,
then $u^*(t,x)\equiv V^*(t)$, where $V^*(t)$ is the unique time almost periodic positive solution of the following
ODE,
\begin{equation}
\label{ode-eq1}
\dot u=uf(t,u).
\end{equation}
\end{theorem}

The following theorem is about the spreading and vanishing dichotomy of \eqref{main-eq}.

\begin{theorem}[Spreading-vanishing dichotomy]
\label{main-thm2}
Assume (H1)-(H5). For any given $h_0>0$ and $u_0(\cdot)$ satisfying  \eqref{initial-value},
the following hold.
\begin{itemize}
\item[(1)]
Either

(i)
$h_{\infty}\le l^*$
and  $\lim_{t\to+\infty} u(t,x;u_0,h_0)=0$

or

(ii)  $h_{\infty}=\infty$ and
$\lim_{t\to\infty}[u(t,x;u_0,h_0)-u^*(t,x)]=0$
locally uniformly for $x\in[0,+\infty)$, where $u^*(t,x)$ is as
in Theorem \ref{main-thm1}.

\item[(2)] If $h_0\ge l^*$, then $h_\infty=\infty$.

\item[(3)]
Suppose $h_{0}<l^*$. Then there exists $\mu^{*}>0$
such that spreading occurs if $\mu>\mu^{*}$ and vanishing occurs if $\mu\le \mu^{*}$.
\end{itemize}
\end{theorem}

We remark that similar results as those in Theorems \ref{main-thm1} and \ref{main-thm2} hold for \eqref{main-doub-eq}
(see Propositions \ref{positive-almost-periodic-solution-doub-prop} and \ref{spreading-vanishing-doub-prop}).

\section{Preliminary}

In this section,
we present some preliminary results to be applied in later sections, including basic properties for principal Lyapunov exponents (see subsection 3.1),
non-increasing property of the so called part metric associated to diffusive KPP equations in both bounded and unbounded domains (see subsection 3.2.),
the asymptotic
dynamics of   diffusive KPP equations  with time almost periodic dependence in fixed bounded environments (see subsection 3.3),  and comparison principles
for free boundary problems (see subsection 3.4).

\subsection{Principal Lyapunov exponents}

Consider \eqref{linearized-eq1}. Let $X=X(l)$, where $X(l)$ is as in \eqref{bounded-domain-space-eq1}. We denote by
$\|\cdot\|$ either the norm in $X$ or in $\mathcal{L}(X,X)$.
Recall that for any $v_0\in X$, \eqref{linearized-eq1} has a unique solution $v(t,\cdot;v_0,a)$ and
$$
\lambda(a,l)=\limsup_{t\to\infty}\frac{\ln\|V(t,a)\|_{X(l)}}{t},
$$
where $V(t,a)v_0=v(t,\cdot;v_0,a)$.
For any $b\in H(a)$, consider also
\begin{equation}
\label{linearized-eq1-1}
\begin{cases}
v_t=v_{xx}+b(t,x)v,\quad 0<x<l\cr
v_x(t,0)=v(t,l)=0.
\end{cases}
\end{equation}
For any $v_0\in X$,
\eqref{linearized-eq1-1} has also a unique solution $v(t,\cdot;v_0,b)$ with $v(0,\cdot;v_0,b)=v_0(\cdot)$.

\begin{lemma}
\label{lyapunov-exponent-lm0}
There is $\phi^l:H(a)\to X^{++}$   satisfying the following properties.
\begin{itemize}
\item[(i)] $\|\phi^l(b)\|=1$ for any $b\in H(a)$ and $\phi^l:H(a)\to X^{++}$ is continuous.

\item[(ii)] $v(t,\cdot;\phi^l(b),b)=\|v(t,\cdot;\phi^l(b),b)\|\phi^l(b(\cdot+t,\cdot))$.

\item[(iii)] $\lim_{t\to\infty}\frac{\ln \|v(t,\cdot;\phi^l(b),b)\|}{t}=\lambda(a,l)$ uniformly in $b\in H(a)$.
\end{itemize}
\end{lemma}

\begin{proof}
It follows from \cite{MiSh1} (see also \cite{MiSh3, ShYi}).
\end{proof}

\begin{lemma}
\label{lyapunov-exponent-lm1}
$\lambda(a,l)$ is a monotone
increasing function of $a$ and $l$.
\end{lemma}

\begin{proof}
For any fixed $a$, suppose
$0<l_{1}\leq l_{2}$. Note that $v(t,\cdot,\phi^{l_1}(a),a)$  and
$v(t,\cdot,\phi^{l_2}(a),a)$ are solutions for the following problems, respectively,
\begin{equation*}
\begin{cases}
\label{linearized-eq3}
v_t=v_{xx}+a(t,x)v, \quad 0<x<l_{1}\cr
v_x(t,0)=v(t,l_{1})=0
\end{cases}
\end{equation*}
and
\begin{equation*}
\begin{cases}
\label{linearized-eq4}
v_t=v_{xx}+a(t,x)v,\quad 0<x<l_{2}\cr
v_x(t,0)=v(t,l_{2})=0.
\end{cases}
\end{equation*}
Choose $0<\epsilon \ll 1$ such that $\phi^{l_2}(a)>\epsilon \phi^{l_1}(a)$ on $[0,l_1]$. Then,  by
 comparison principle for parabolic equations, we have that
$$v(t,x;\epsilon\phi^{l_1}(a),a)<v(t,x;\phi^{l_2}(a),a)\quad \forall \,\, 0\le x\le l_1.$$
By Lemma \ref{lyapunov-exponent-lm0} and a priori estimates for parabolic equations, we have that
\begin{equation*}
\begin{split}
\lambda(a,l_{2})=&\lim_{t\rightarrow\infty}\frac{\ln\|v(t,\cdot,\phi^{l_2}(a),a)\|_{X(l_2)}}{t} \\
\geq&\lim_{t\rightarrow\infty}\frac{\ln\|v(t,\cdot,\epsilon\phi^{l_1}(a),a)\|_{X(l_1)}}{t} \\
=&\lim_{t\rightarrow\infty}\frac{\ln\epsilon + \ln\|v(t,\cdot,\phi^{l_1}(a),a)\|_{X(l_1)}}{t} \\
=&\lambda(a,l_{1})
\end{split}
\end{equation*}
Thus, $\lambda(a,l)$ is a monotone increasing function of $l$.

If we fix $l$, we can use comparison principle and a priori estimates for parabolic
equations again to get that $\lambda(a,l)$ is a monotone increasing function of $a$.
\end{proof}

In the following, if no confusion occurs, we will write $\phi^l(b)$ as $\phi(b)$.

\begin{lemma}
\label{lyapunov-exponent-lm2} $\lambda(a,l)$ is continuous in $a$.
\end{lemma}

\begin{proof}
For any $k\ge 1$, Consider the following problem
\begin{equation}
\begin{cases}
\label{linearized-eq5}
v_t=v_{xx}+v(a(t,x)\pm \frac{1}{k}),\quad 0<x<l\cr
v_x(t,0)=v(t,l)=0.
\end{cases}
\end{equation}
Note that
$e^{\pm \frac{1}{k}t}v(t,x;\phi(a),a)$ is a solution of
\eqref{linearized-eq5}. It follows from Lemma \ref{lyapunov-exponent-lm0} that
\begin{equation*}
\begin{split}
\lambda(a\pm \frac{1}{k},l)=&\lim_{t\to\infty}
\frac{\ln \|v(t,\cdot;\phi(a),a)\|}{t}\pm \frac{1}{k} \\
=&\lambda(a,l)\pm \frac{1}{k}.
\end{split}
\end{equation*}
Let $k\to\infty$, we can get that
$\lambda(a\pm \frac{1}{k},l)-\lambda(a,l)\to 0$.
This together with Lemma \ref{lyapunov-exponent-lm1} implies that
 $\lambda(a,l)$ is continuous in $a$.
\end{proof}

\begin{lemma}
\label{lyapunov-exponent-lm3} Suppose that $a(t,x)\equiv a(t)$. Then
$$
\lambda(a,l)=\hat a+\lambda_0(l),
$$
where $\hat a=\lim_{t\to\infty}\frac{1}{t}\int_0^t a(s)ds$ and $\lambda_0(l)$ is the principal
eigenvalue of
\begin{equation}
\label{egv-eq}
\begin{cases}
u_{xx}=\lambda u,\quad 0<x<l\cr
u_x(0)=u(l)=0.
\end{cases}
\end{equation}
\end{lemma}

\begin{proof}
Let $\tilde v(t,x)=v(t,x)e^{-\int_0^t a(s)ds}$.  Then \eqref{linearized-eq1} becomes
$$
\begin{cases}
\tilde v_t=\tilde v_{xx},\quad 0<x<l\cr
\tilde v_x(t,0)=\tilde v(t,l)=0.
\end{cases}
$$
It then follows that $\lambda(a,l)=\hat a+\lambda(0,l)$. It is clear that
$\lambda(0,l)=\lambda_0(l)$. The lemma then follows.
\end{proof}

\begin{remark}
\label{Lyapunov-exp-rk}
 Let $a(t,x)$ be a given $C^1$ function which is  almost periodic function in $t$ uniformly in $x$ in bounded sets and $\gamma\in\RR$.
Consider
\begin{equation}
\label{aux-linearized-eq1}
\begin{cases}
v_t=v_{xx}+\gamma v_x+a(t,x)v,\quad 0<x<l\cr
v(t,0)=v(t,l)=0.
\end{cases}
\end{equation}
Let $\tilde X(l)$ be as in \eqref{bounded-domain-space-eq2}.
Then, for any $v_0\in \tilde X(l)$, \eqref{aux-linearized-eq1} has a unique solution
$v(t,\cdot;v_0,a)$ with $v(0,\cdot;v_0,a)=v_0(\cdot)$. Let
$$
\tilde\lambda(a,\gamma,l)=\limsup_{t\to\infty}\frac{\ln\|V(t,a)\|_{\tilde X(l)}}{t}
$$
where $V(t,a)v_0=v(t,\cdot;v_0,a)$.
$\tilde\lambda(a,\gamma,l)$ is called the {\it principal Lyapunov exponent} of \eqref{aux-linearized-eq1}.
Principal Lyapunov exponent theory for \eqref{linearized-eq1} also holds for \eqref{aux-linearized-eq1}.
In particular, $\tilde\lambda(a,\gamma,l)$ is continuous in $a$ and $\gamma$.

\end{remark}

\subsection{Part metric associated to diffusive KPP equations}

In this subsection, we present the non-increasing property of the so called part metric associated to \eqref{aux-main-eq1},
and the following
diffusive KPP equations  with time almost periodic dependence in fixed bounded domain,
\begin{equation}
\label{fixed-boundary-main-eq1} \begin{cases} u_t=u_{xx}+uf(t,x,u),\quad
0<x<l\cr u_x(t,0)=u(t,l)=0.
\end{cases}
\end{equation}
Throughout this subsection, we assume (H1) and (H2).
Let
$$
H(f)={\rm cl}\{f(\cdot+\tau,\cdot,\cdot)\,|\, \tau\in\RR\},
$$
where the closure is taken in the open compact topology. Observe that for any $g\in H(f)$, $g$ also satisfies (H1) and (H2).

First, consider \eqref{fixed-boundary-main-eq1}. For given $g\in H(f)$, we also consider
\begin{equation}
\label{fixed-boundary-eq3} \begin{cases} u_t=u_{xx}+ug(t,x,u),\quad
0<x<l\cr u_x(t,0)=u(t,l)=0
\end{cases}
\end{equation}
 Let $X(l)$ be as in \eqref{bounded-domain-space-eq1}.
By semigroup theory, for any $g\in H(f)$ and $u_0\in X(l)$, \eqref{fixed-boundary-eq3}
has a unique (local) solution $u(t,x;u_0,g)$ with $u(0,x;u_0,g)=u_0(x)$.
Note that $u(t,x;u_0,s):=u(t-s,x;u_0,f(\cdot+s,\cdot,\cdot))$ is the solution of
\eqref{fixed-boundary-main-eq1} with $u(s,x;u_0,s)=u_0(x)$. By (H1) and comparison principle for
parabolic equations, if $u_0\in X^+(l)$, then $u(t,\cdot;u_0,g)$ exists and $u(t,\cdot;u_0,g)\in X^+(l)$ for all $t>0$.
Moreover, if $u_0\in X^+(l)\setminus\{0\}$, then $u(t,\cdot;u_0,g)\in X^{++}(l)$ for $t>0$.

For any $u_1,u_2\in X^{++}(l)$, we can define the so called part metric, $\rho(u_1,u_2)$, between $u_1$ and $u_2$, as follows,
\begin{equation}
\label{part-metric-eq1}
\rho(u_1,u_2)=\inf\{\ln\alpha\,|\, \alpha\ge 1,\,\, \frac{1}{\alpha}u_1(\cdot)\le u_2(\cdot)\le\alpha u_1(\cdot)\}.
\end{equation}
Note that if $u_1,u_2\in X^{++}(l)$, then $u(t,\cdot;u_i,g)\in X^{++}(l)$ $(i=1,2$) for any $t>0$ and $g\in H(f)$. Hence
$\rho(u(t,\cdot;u_1,g),u(t,\cdot;u_2,g))$ is also well defined.

Next, consider  \eqref{aux-main-eq1} and  consider also
\begin{equation}
\label{unbounded-eq3}
\begin{cases}
u_t=u_{xx}+ug(t,x,u),\quad 0<x<\infty\cr
u_x(t,0)=0
\end{cases}
\end{equation}
for all $g\in H(f)$.

Let
\begin{equation}
\label{tidle-x-space}
\tilde X=\{u\in C([0,\infty))\,|\, u\quad \text{is uniformly continuous and bounded on}\,\, \, [0,\infty)\}
\end{equation}
with norm $\|u\|=\sup_{x\in[0,\infty)}|u(x)|$ and
$$
\tilde X^+=\{u\in \tilde X \,|\, u(x)\ge 0\},
$$
$$
\tilde X^{++}=\{u\in \tilde X \,|\, \inf u(x)>0\}.
$$
Note that $\tilde X^{++}$ is not empty and is an open subset of $\tilde X^+$.
By semigroup theory (see \cite{PA}), for any $g\in H(f)$ and $u_0\in \tilde X$, \eqref{unbounded-eq3}
has a unique solution $u(t,x;u_0,g)$ with $u(0,x;u_0,g)=u_0(x)$.
By (H1) and comparison principle for parabolic equations, if $u_0\in \tilde X^+$, then
$u(t,\cdot;u_0,g)$ exists and $u(t,\cdot;u_0,g)\in \tilde X^+$ for all $t>0$. Moreover, if $u_0\in\tilde X^{++}$, then
$u(t,\cdot;u_0,g)\in \tilde X^{++}$ for all $t>0$.

For given $u_1,u_2\in \tilde X^{++}$, we can also define the part metric, $\rho(u_1,u_2)$, between $u_1$ and $u_2$ as follows,
$$
\rho(u_1,u_2)=\inf\{\ln\alpha\,|\, \alpha\ge 1, \,\, \frac{1}{\alpha}u_1(\cdot)\le u_2(\cdot)\le \alpha u_1(\cdot)\}.
$$
Note that if $u_1,u_2\in\tilde X^{++}$, then $\rho(u(t,\cdot;u_1,g),u(t,\cdot;u_2,g))$ is well defined for $t>0$.

We now have the following proposition about the non-increasing of part metric.

\begin{proposition}
\label{part-metric-prop}
\begin{itemize}

\item[(1)] Consider \eqref{fixed-boundary-eq3} and let $u(t,\cdot;u_0,g)$ denote the solution of \eqref{fixed-boundary-eq3}
with $u(0,\cdot;u_0,g)=u_0(\cdot)\in X(l)$.
 For given $u_0,v_0\in X^{++}(l)$ with $u_0\not = v_0$, $\rho(u(t,\cdot;u_0,g),u(t,\cdot;v_0,g))$ is strictly decreasing as $t$
increases.

\item[(2)] Consider \eqref{unbounded-eq3} and let $u(t,\cdot;u_0,g)$ denote the solution of \eqref{unbounded-eq3}
with $u(0,\cdot;u_0,g)\in \tilde X$.
\item[] (i) Given any $u_0,v_0\in \tilde X^{++}$ and $g\in H(f)$, $\rho(u(t,\cdot;u_0,g),u(t,\cdot;v_0,g))$
decreases as $t$ increases.

\item[] (ii)  For any $\epsilon>0$, $\sigma>0$, $M>0$,  and $\tau>0$  with $\epsilon<M$ and
$\sigma\le \ln \frac{M}{\epsilon}$, there is $\delta>0$   such that
for any $g\in H(f)$, $u_0,v_0\in \tilde X^{++}$ with $\epsilon\le u_0(x)\le M$, $\epsilon\le v_0(x)\le M$ for $x\in\RR^{+}$ and
$\rho(u_0,v_0)\ge\sigma$, there holds
$$
\rho(u(\tau,\cdot;u_0,g),u(\tau,\cdot;v_0,g))\le \rho(u_0,v_0)- \delta.
$$
\end{itemize}
\end{proposition}

\begin{proof}   The proposition can be proved by the similar arguments as in \cite[Proposition 3.4]{KoSh}.
For the completeness, we provide a proof in the following.

 (1)  For any $u_0,v_0\in  X^{++}(l)$ with $u_0\not = v_0$,  there is $\alpha^*> 1$ such that  $\rho(u_0,v_0)=\ln \alpha^{*}$
and $\frac{1}{\alpha^{*}}u_0\leq v_0\leq \alpha^{*} u_0$.
By comparison principle for parabolic equations,
$$u(t, \cdot;v_0,g)\le u(t, \cdot; \alpha^{*}u_0,g)\quad {\rm for}\quad t>0.
$$
Let
 $$v(t, x)=\alpha^{*}u(t, x; u_0,g).$$
We then have
\begin{align*}
v_{t}(t,x)&=v_{xx}(t,x)+v(t,x)g(t,x, u(t,x;u_0,g))\\
& =
v_{xx}(t,x)+v(t,x)g(t, x,v(t,x))+v(t,x)g(t, x,u(t,x;u_0,g))-v(t,x)g(t,x,v(t,x))\\
&> v_{xx} (t,x)+ v(t,x)g(t, x,v(t,x))\quad \ for \ all  \ t>0,\quad 0\le x<l,
\end{align*}
and
$$
v_x(t,0)=0,\quad v(t,l)=0 \  \ {\rm for} \ {\rm all} \ t>0.
$$
By strong comparison principle  for parabolic equations,
$$
u(t,x;\alpha^* u_0,g)< \alpha^* u(t,x;u_0,g)\quad {\rm for}\quad 0\le x < l.
$$
Then by Hopf lemma for parabolic equations, there is $\tilde\alpha^*<\alpha^*$ such that
$$
u(t,x;\alpha^* u_0,g)\le \tilde \alpha^* u(t,x;u_0,g)\quad {\rm for}\quad 0\le x\le l
$$
and hence
$$
u(t,\cdot;v_0,g)\leq \tilde \alpha^* u(t,\cdot;u_0,g)
$$
for $t>0$.
Similarly, we can prove that
$$
\frac{1}{\bar \alpha^*}u(t,\cdot;u_0,g)\le u(t,\cdot;v_0,g)
$$
for some $\bar \alpha^*<\alpha^*$ and  $t>0$. It then follows that
$$
\rho(u(t,\cdot;u_0,g),u(t,\cdot;v_0,g))< \rho(u_0,v_0)\quad \ for \ all \,\, t\ge 0
$$
and then
$$
\rho(u(t_2,\cdot;u_0,g),u(t_2,\cdot;v_0,g))< \rho(u(t_1,\cdot;u_0,g),u(t_1,\cdot;v_0,g))
\quad \ for \ all \,\, 0\le t_1< t_2.
$$

(2) (i) It follows from the arguments in (1).

(ii)
Let $\epsilon>0$, $\sigma>0$, $M>0$, and $\tau>0$ be given and $\epsilon<M$, $\sigma<\ln \frac{M}{\epsilon}$. First,  we claim that
there are $\epsilon_1>0$ and $M_1>0$ such that
for any $g\in H(f)$ and $u_0\in \tilde{X}^{++}$ with $\epsilon\le u_0(x)\le M$ for $x\in\RR^+$, there holds
\begin{equation*}
\label{part-metric-eq1}
\epsilon_1\le u(t,x;u_0,g)\le M_1\quad \ for \ all \,\, t\in[0,\tau],\,\, x\in \RR^+.
\end{equation*}
In fact, let $\tilde M>0$ be such that  $f(t,x,u)<0$ for $u\ge \tilde M$. Then for $0<\tilde\epsilon<\max\{\epsilon,\tilde M\}$,
$u(t,\cdot;u_{\tilde \epsilon},g)\le \tilde{M}$ for all $t\ge 0$ and $g\in H(f)$, where $u_{\tilde\epsilon}(x)\equiv \tilde \epsilon$.
Note that $g(t,x,u)\ge \tilde \alpha =\inf_{t\in\RR,x\in\RR^+}f(t,x,\tilde M)$ for $u\le \tilde M$.
Hence by comparison principal for parabolic equations,
$$
\tilde M\ge u(t,x;u_{\tilde\epsilon},g)\ge e^{\tilde{\alpha} t} \tilde \epsilon\quad \ for \ all \,\, t\ge 0,\quad x\in \RR^+.
$$
The claim then follows.

Let
\begin{equation*}
\label{part-metric-eq2}
\delta_1=\epsilon_1^2 e^\sigma (1-e^\sigma)\sup_{g\in H(f),t\in[0,\tau],x\in \RR^+,u\in[\epsilon_1,M_1M/\epsilon]}g_u(t,x,u).
\end{equation*}
Then $\delta_1>0$ and there is $0<\tau_1\le\tau$ such that
\begin{equation}
\label{part-metric-eq3-1}
\frac{\delta_1}{2}\tau_1<e^\sigma\epsilon_1
\end{equation}
and
\begin{equation}
\label{part-metric-eq3-2}
\Big|\frac{\delta_1}{2}tv g_u(t,x,w)\Big|+\Big|\frac{\delta_1}{2}tg(t,x,v-\frac{\delta_1}{2}t)\Big|\le\frac{\delta_1}{2}
\end{equation}
for all  $g\in H(f),\,\,  t\in [0,\tau_1],\,\, x\in \RR^+,\,\, v,w\in[0,M_1M/\epsilon]$.
Let
\begin{equation}
\label{part-metric-eq4}
\delta_2=\frac{\delta_1\tau_1}{2 M_1}.
\end{equation}
Then $\delta_2<e^\sigma$ and $0<\frac{\delta_2 \epsilon}{M}<1$. Let
\begin{equation}
\label{part-metric-eq5}
\delta=-\ln\big(1-\frac{\delta_2\epsilon}{M}\big).
\end{equation}
Then $\delta>0$. We prove that  $\delta$ defined in \eqref{part-metric-eq5}  satisfies the property in the proposition.

For any $u_0,v_0\in\tilde{X}^{++}$ with $\epsilon\le u_0(x)\le M$ and $\epsilon\le v_0(x)\le M$ for $x\in \RR^+$ and
$\rho(u_0,v_0)\ge\sigma$,  there is $\alpha^*> 1$ such that  $\rho(u_0,v_0)=\ln \alpha^{*}$
and $\frac{1}{\alpha^{*}}u_0\leq v_0\leq \alpha^{*} u_0$.

Note that $e^\sigma\le\alpha^*\le \frac{M}{\epsilon}$. Let
 $$v(t, x)=\alpha^{*}u(t, x; u_0,g)$$
\begin{align*}
v_{t}(t,x)&=v_{xx}(t,x)+v(t,x)g(t,x, u(t,x;u_0,g))\\
& =
v_{xx}(t,x)+v(t,x)g(t, x,v(t,x))+v(t,x)g(t, x,u(t,x;u_0,g))-v(t,x)g(t,x,v(t,x))\\
&\ge  v_{xx} (t,x)+ v(t,x)g(t, x,v(t,x))+\delta_1\quad \forall 0<t\le\tau_1,\,\, x\in \RR^+.
\end{align*}  This together with \eqref{part-metric-eq3-1}, \eqref{part-metric-eq3-2}
implies that
$$
(v(t,x)-\frac{\delta_1}{2}t)_t\ge \big(v(t,x)-\frac{\delta_1}{2}t\big)_{xx}+\big(v(t,x)-\frac{\delta_1}{2}t\big)g\big(t,x,v(t,x)-\frac{\delta_1}{2}t\big)
$$
for $0<t\le \tau_1$.
Then by comparison principle for parabolic equations,
$$
u(t,\cdot;\alpha^* u_0,g)\leq \alpha^* u(t,\cdot; u_0,g)-\frac{\delta_1}{2}t\quad {\rm for}\quad 0<t\le\tau_1.
$$
By \eqref{part-metric-eq4},
$$
u(\tau_1,\cdot;v_0,g)\le (\alpha^*-\delta_2) u(\tau_1,\cdot;u_0,g).
$$
Similarly, it can be proved that
$$
\frac{1}{\alpha^*-\delta_2}u(\tau_1,\cdot;u_0,g)\le u(\tau_1,\cdot;v_0,g).
$$
It then follows that
$$
\rho(u(\tau_1,\cdot;u_0,g),u(\tau_1,\cdot;v_0,g))\le \ln(\alpha^*-\delta_2) =\ln\alpha^*+\ln (1-\frac{\delta_2}{\alpha^*})\le \rho(u_0,v_0)-\delta.
$$
and hence
$$
\rho(u(\tau,\cdot;u_0,g),u(\tau,\cdot;v_0,g))\le\rho(u(\tau_1,\cdot;u_0,g),u(\tau_1,\cdot;v_0,g))
\le\rho(u_0,v_0)-\delta.
$$
\end{proof}


\subsection{Asymptotic dynamics of diffusive KPP equations  with time almost periodic dependence on fixed bounded domain}

In this subsection, we consider the asymptotic dynamics of \eqref{fixed-boundary-main-eq1}.
 Throughout this section, we assume that $f$ satisfies (H1) and (H2).

Let $X(l)$ be as in \eqref{bounded-domain-space-eq1} and $u(t,\cdot;u_0,g)$ be the solution of
\eqref{fixed-boundary-eq3}
with $u(0,\cdot;u_0,g)=u_0(\cdot)$.
Observe that \eqref{fixed-boundary-eq3} generates a skew-product
hemodynamical system,
$$\Pi_{t}: X^{+}(l)\times H(f)\to X^{+}(l)\times H(f)$$
of the following form:
$$\Pi_{t}(u_{0},g)=(u(t,\cdot,u_{0},g), g_{t})\quad \forall(u_{0},g)\in X^{+}(l)\times H(f).$$
The system $\Pi_{t}$ is strongly monotone in the sense that
$u(t,\cdot;u_{0},g)\ll u(t,\cdot;v_{0},g)$ for any $0\leq u_{0}\leq v_{0}$
with $u_{0}\neq v_{0}$ and any $t>0$, where we write $u_{0}\ll v_{0}$ if
$v_{0}-u_{0}\in X^{++}(l)$.

\begin{proposition}
\label{fixed-boundary-prop1}
Let $a(t,x)=f(t,x,0)$.
\begin{itemize}
\item[(1)] If $\lambda(a,l)< 0$, then for any $u_0\in X^+(l)$,
$\|u(t,\cdot;u_0,g)\|\to 0$ as $t\to\infty$ uniformly in $g\in H(f)$. In particular,
$\|u(s+t,\cdot;u_0,s)\|\to 0$ as $t\to\infty$ uniformly in
$s\in\RR$.

\item[(2)] If $\lambda(a,l)>0$, then
there exists $u^l:H(f)\to X^{++}(l)$ such that $u^l(g)$ is continuous in $g\in H(f)$,
$u(t,\cdot;u^l(g),g)=u^l(g\cdot t)(\cdot)$ for any $g\in H(f)$, and for any $u_0\in X^+(l)\setminus\{0\}$,
$$
\|u(t,\cdot;u_0,g)-u(t,\cdot;u^l(g),g)\|\to 0
$$
as $t\to\infty$ uniformly in $g\in H(f)$. In particular,
$u^{*,l}(t,x):=u(t,x;u^l(f),f)$ is almost periodic in $t\in\RR$ and for any $u_0\in X^+(l)\setminus\{0\}$,
$$
\|u(s+t,\cdot;u_0,s)-u^{*,l}(s+t,\cdot)\|\to 0
$$
as $t\to\infty$ uniformly in $s\in\RR$, where $u(s+t,\cdot;u_0,s)=u(t,\cdot;u_0,f(\cdot+s,\cdot,\cdot))$.
\end{itemize}
\end{proposition}

The proposition follows from \cite[Theorem A]{MiSh2}. For completeness,
we provide a proof in the following.

\begin{proof}[Proof of Proposition \ref{fixed-boundary-prop1}]
(1) Let $b(t,x)=g(t,x,0)$ for any $g\in H(f)$. Since $\lambda(a,l)<0$,
it is well known that $\|v(t,x;\phi^l(b),b)\|\to 0$ as $t\to\infty$, where $v(t,x;\phi^l(b),b)$ is
the solution of \eqref{linearized-eq1-1} with $v(0,\cdot;\phi^l(b),b)=\phi^l(b)$ and
$\phi^l(b)$ is as in Lemma \ref{lyapunov-exponent-lm0}.
For any $u_{0}\in X^{+}(l)$, we can choose $M>0$ such that $u_{0}\leq M\phi^l(b)$
for $x\in(0,l)$. It follows from comparison principle for parabolic equations
that
$0\leq u(t,x,u_{0},g)\leq Mv(t,x,\phi^l(b),b)$ for $x\in[0,l]$.
This together with a priori estimates for parabolic equations implies that
 $\|u(t,\cdot,u_{0},g)\|\to 0$
as $t\to\infty$.

(2) Choose $\xi>0$ such that $\lambda(a,l)-\xi>0$. Let
$\bar{a}(t,x)=f(t,x,0)-\xi$
and
$\phi^l: H(\bar a)\to X^{++}(l)$ be as in Lemma \ref{lyapunov-exponent-lm0}.
For any $g\in H(f)$, we choose $\bar{b}\in H(\bar{a})$ such that
$\bar{b}(t,x)=g(t,x,0)-\xi$. Let
$v(t,x;\phi^l(\bar{b}),\bar{b})$ be the solution of
\begin{equation*}
\begin{cases}
\label{linearized-eq7}
v_t=v_{xx}+\bar{b}(t,x)v,\quad  0<x<l \cr v_x(t,0)=v(t,l)=0
\end{cases}
\end{equation*}
with $v(0,x;\phi^l(\bar b),\bar b)=\phi^l(\bar b)(x)$.

Since $\lambda(\bar a,l)=\lambda(a,l)-\xi>0$, we can find $T>0$, such that
$$\|v(T,\cdot;\phi^l(\bar{b}),\bar{b})\|\geq1 \, \,\,\forall\,\, \bar{b}\in H(\bar{a}).$$
Choose $0<\epsilon\ll1$ such that for any $g\in H(f)$,
$$ug(t,x,u)\geq(g(t,x,0)-\xi)u \quad {\rm for}\,\,
0\leq u\leq \sup_{0\leq t\leq T}\|v(t,\cdot;\epsilon\phi^l(\bar{b}),\bar{b})\|.$$
Using comparison principle for parabolic equations we obtain that
$v(t,\cdot;\epsilon\phi^l(\bar{b}),\bar{b})  (0\leq t\leq T)$ is a subsolution of
the problem \eqref{fixed-boundary-eq3}. We then have
$$
u(T,\cdot;\epsilon \phi^l(\bar b),g)\ge \epsilon \phi^l(\bar b_{T})
$$
and then
$$
u(nT,\cdot;\epsilon \phi^l(\bar b),g)\ge \epsilon \phi^l(\bar b_{nT}).
$$
Let  $\omega(\epsilon\phi^l(\bar a),f)$ be the $\omega$-limit set of $\Pi_t(\epsilon \phi^l(\bar a),f)$. We then have
$\omega(\epsilon \phi^l(\bar a),f)\subset X^{++}(l)\times H(f)$.

We claim that for any $g\in H(f)$,
there is unique $u^l(g)\in X^{++}(l)$ such that $(u^l(g),g)\in\omega(\epsilon \phi^l(\bar a),f)$.
In fact, if there is $g\in H(f)$ such that there are $u_1,u_2\in X^{++}(l)$ with
$(u_i,g)\in \omega(\epsilon \phi^l(\bar a),f)$ and $u_1\not = u_2$, then $(u(t,\cdot;u_i,g),g_t)\in \omega(\epsilon \phi^l(\bar a),f)$ for all
$t\in\RR$. By Proposition \ref{part-metric-prop}(1), there is $\rho_\infty>0$ such that
$\rho(u(t,\cdot;u_1,g),u(t,\cdot;u_2,g))\to \rho_\infty$ as $t\to -\infty$. For any $t_n\to -\infty$, without loss of
generality, assume that
$g_{t_n}\to g^*$ and $u(t_n\cdot;u_i,g)\to u_i^*$. Then
$$
u(t,\cdot;u_i^*,g^*)=\lim_{n\to\infty}u(t+t_n,\cdot;u_i,g)
$$
and
$$
\rho(u(t,\cdot;u_1^*,g^*),u(t,\cdot;u_2^*,g^*))=\lim_{n\to\infty}\rho(u(t+t_n,\cdot;u_1,g),u(t+t_n,\cdot;u_2,g))=\rho_\infty
$$
for all $t\in\RR$, which contradicts to Proposition \ref{part-metric-prop}(1). Therefore,
the claim holds and $u^l:H(f)\to X^{++}$ is continuous. In particular,
$u^{*,l}(t,x)=u(t,x;u^l(f),f)$ is an almost periodic solution.
Moreover, by the above arguments, for any $u_0\in X^{++}$,
$\omega(u_0,f)=\omega(\epsilon \phi^l(\bar a),f)$ and then
$$
\lim_{t\to\infty}\|u(t,\cdot;u_0,f)-u^{*,l}(t,\cdot)\|=0.
$$
\end{proof}

\subsection{Comparison principal for free boundary problems}

In order for later application, we need a comparison principle which can
be used to estimate both $u(t,x)$ and the free boundary $x=h(t)$.
\begin{proposition}
\label{comparison-principle} Suppose that $T\in (0,\infty)$, $\bar{h}
\in C^{1}([0,T])$, $\bar{u}\in C(\bar{D}^{*}_{T})\cap C ^{1,2}(D^{*}_{T})$
with $D^{*}_{T}=\{(t,x)\in\R^{2}:0<t\leq T, 0<x<\bar{h}(t)\}$, and
\begin{equation*}
\begin{cases}
\bar{u}_t\geq\bar{u}_{xx}+\bar{u}f(t,x,\bar{u}),\quad &t>0, 0<x<\bar{h}(t)
\cr \bar{h}^{'}(t)\geq-\mu\bar{u}_x(t,\bar{h}(t)),\quad &t>0, \cr
\bar{u}_x(t,0)\leq 0,u(t,\bar{h}(t))=0,\quad  &t>0.
\end{cases}
\end{equation*}
If $h_{0}\leq \bar{h}(0)$ and $u_{0}(x)\leq\bar{u}(0,x)$ in $[0,h_{0}]$,
then the solution $(u,h)$ of the free boundary problem \eqref{main-eq}
satisfies
$$h(t)\leq \bar{h}(t)  \ for\ all\ t\in(0,T],\ u(t,x)\leq\bar{u}(x,t) \ for \
t\in(0,T]\ and\ x\in(0,h(t)).$$

\begin{proof}
The proof of this Proposition is similar to that of Lemma 3.5 in
\cite{DuLi} and Lemma 2.6 in \cite{DuGu}.
\end{proof}

\begin{remark}
The pair $(\bar{u},\bar{h})$ in Proposition \ref{comparison-principle}
is called an upper solution of the free boundary problem. We can
define a lower solution by reversing all the inequalities in the
obvious places.

\end{remark}

\end{proposition}

\begin{proposition}
\label{comparison-principle1} Suppose that $T\in (0,\infty)$, $\bar{g}, \bar{h}
\in C^{1}([0,T])$, $\bar{u}\in C(\bar{D}^{*}_{T})\cap C ^{1,2}(D^{*}_{T})$
with $D^{*}_{T}=\{(t,x)\in\R^{2}:0<t\leq T, \bar{g}(t)<x<\bar{h}(t)\}$, and
\begin{equation*}
\begin{cases}
\bar{u}_t\geq\bar{u}_{xx}+\bar{u}f(t,x,\bar{u}),\quad &t>0, \bar{g}(t)<x<\bar{h}(t)
\cr \bar{u}(t, \bar{h}(t))=0, \bar{h}^{'}(t)\geq-\mu\bar{u}_x(t,\bar{h}(t)),\quad &t>0 \cr
\bar{u}(t, \bar{g}(t))=0, \bar{g}^{'}(t)\leq -\mu\bar{u}_x(t,\bar{g}(t)),\quad  &t>0.
\end{cases}
\end{equation*}
If $[g_0, h_0] \subset [\bar{g}(0), \bar{h}(0)]$ and $u_{0}(x)\leq\bar{u}(0,x)$ in $[g_0,h_0]$,
then the solution $(u,g,h)$ of the free boundary problem \eqref{main-doub-eq} satisfies
$$g(t)\ge \bar{g}(t), h(t)\leq \bar{h}(t)  \ for\ all\ t\in(0,T],\ u(t,x)\leq\bar{u}(x,t) \ for \
t\in(0,T]\ and\ x\in(g(t),h(t)).$$

\begin{proof}
The proof of this Proposition only requires some obvious modifications as in
Proposition \ref{comparison-principle}.
\end{proof}

\end{proposition}

\begin{proposition}
\label{global-existence} For any given $h_0>0$ and $u_0\ge
0$, $(u(t,x;u_0,h_0), h(t;u_0,h_0))$ exists for all $t\ge 0$.
\end{proposition}

\begin{proof}
The proof is similar to that of Theorem 4.3 in \cite{DuGu}.
\end{proof}

\begin{remark}
From the uniqueness of the solution to \eqref{main-eq} and
some standard compactness argument, we can obtain that the
unique solution $(u,h)$ depends continuously on $u_{0}$ and
the parameters appeared in \eqref{main-eq}.
\end{remark}

\section{Asymptotic Dynamics of Diffusive KPP Equations on
 Fixed Unbounded Domain and Proof of Theorem \ref{main-thm1}}

In this section, we consider the asymptotic dynamics of \eqref{aux-main-eq1} and
prove Theorem \ref{main-thm1}. Throughout this section, we assume that  $f$ satisfies (H1)-(H5).
We let $\tilde X$ be as in \eqref{tidle-x-space} and
$u(t,\cdot;u_0,g)$ be the solution of \eqref{unbounded-eq3} with $u(0,\cdot;u_0,g)=u_0(\cdot)\in\tilde X$.
The main results of this section are stated in the following proposition.

\begin{proposition}
\label{unbounded-prop1}
Assume (H1)-(H5). There is $u^*:H(f)\to \tilde X^{++}$ satisfying the following properties.
\begin{itemize}
\item[(1)] (Almost periodicity in time)  $u^*(g)(x)$ is continuous in $g\in H(f)$ in open compact topology with respect to $x$
(that is, if $g_n\to g$ in $H(f)$, then $u^*(g_n)(x)\to u^*(g)(x)$ locally uniformly in $x$)
and  $u(t,x;u^*(g),g)=u^*(g\cdot t)(x)$ (hence $u^*(g\cdot t)(x)$ is an almost periodic solution of \eqref{unbounded-eq3}).

\item[(2)] (Stability) For any $u_0\in\tilde X^{++}$,
$$
\|u(t,\cdot;u_0,g)-u^*(g(\cdot+t,\cdot,\cdot))(\cdot)\|_{\tilde X}\to 0
$$
as $t\to\infty$ uniformly in $g\in H(f)$.

\item[(3)] (Uniqueness) For given $g\in H(f)$, if $\tilde u^{*}(t,x)$ is an entire positive solution of \eqref{unbounded-eq3}, and $\inf_{t\in\RR,x\in\RR^+}\tilde u^{*}(t,x)>0$, then $\tilde u^{*}(t,x)=u(t,x;u^*(g),g)$.

 \item[(4)] (Spatial homogeneity) If $f(t,x,u)\equiv f(t,u)$, then $u^*(g)(x)$ is independent of $x$ and
 $V^*(t;g)=u^*(g\cdot t)$ is the unique time almost periodic solution of
 \begin{equation}
 \label{ode-eq2}
 u_t=ug(t,u).
 \end{equation}
\end{itemize}
\end{proposition}

\begin{proof} [Proof of  Theorem \ref{main-thm1}]
Let $u^*(t,x)=u^*(f \cdot t)(x)$, where $u^*(f \cdot t)$ is as in
 Proposition \ref{unbounded-prop1}. Theorem \ref{main-thm1} then follows.
 \end{proof}

 We  remark that
 the existence and uniqueness of positive solutions which are bounded away from $0$
of KPP equations in heterogeneous unbounded domains have been studied in \cite{BeHaNa} (see \cite[Propositions 1.7, 1.8]{BeHaNa}).
The almost periodicity and stability results in Proposition \ref{unbounded-prop1} are new.

To prove Proposition \ref{unbounded-prop1}, we first prove two lemmas.

For any $L\ge L^*$ and $y\ge y^*$, consider
\begin{equation}
\label{unbounded-domain-eq1}
\begin{cases}
u_t=u_{xx}+u g^y(t,x,u),\quad 0<x<L\cr
u(t,0)=u(t,L)=0,
\end{cases}
\end{equation}
where $ g^y(t,x,u)=g(t,x+y,u)$ for $0\le x\le L$.
By (H4),
$\tilde \lambda(g^y(\cdot,\cdot,0),L)>0$ for $y\ge y^*$. Then by the arguments of Proposition \ref{fixed-boundary-prop1},
\eqref{unbounded-domain-eq1} has a unique time almost periodic
 positive solution $u^*(t,x;g,y,L)$.
Note that
$$
u^*(t,x;g,y,L)=u^*(0,x;g\cdot t,y,L).
$$

\begin{lemma}
\label{unbounded-domain-lm2}
Assume (H1)-(H5). Fix a $L\ge L^*$. Then
\begin{equation}
\label{bounded-away-from-zero-eq1}
\inf_{y\ge y^*, L/4\le x\le 3L/4, g\in H(f)}u^*(0,x;g,y,L)>0.
\end{equation}

\end{lemma}

\begin{proof}
Assume that \eqref{bounded-away-from-zero-eq1} does not hold. Then there are $y_n\ge y^*$, $g_n\in H(f)$, and $x_n\in [L/4,3L/4]$ such that
$$
\lim_{n\to\infty} u^*(0,x_n;g_n,y_n,L)=0.
$$
By (H5), without loss of generality, we may assume that
$$
g_n^{y_n}(t,x,u)\to g^*(t,x,u)
$$
uniformly in $t\in\RR$ and $(x,u)$ in bounded sets.

Let $a_n(t,x)=g_n^{y_n}(t,x,0)$ and $a^*(t,x)=g^*(t,x,0)$. By (H2), $g^*$ is almost periodic in $t$.
Then
$$
\tilde \lambda(a_n,L)\to\tilde  \lambda(a^*,L)
$$
and hence $\tilde \lambda(a^*,L)>0$. Note that  for any $\epsilon>0$,
$$
g_n^{y_n}(t,x,u)\ge g^*(t,x,u)-\epsilon\quad \forall\,\, n\gg 1,\,\, 0\le x\le L.
$$
and
\begin{equation*}
\label{unbounded-domain-eq2}
\begin{cases}
u_t=u_{xx}+u\big( g^*(t,x,u)-\epsilon\big),\quad 0<x<L\cr
u(t,0)=u(t,L)=0
\end{cases}
\end{equation*}
has a unique positive almost periodic solution $\tilde u^{L}(t,x)$ with
$\inf_{ L/4\le x\le  3L/4,t\in\RR} \tilde u^{L}(t,x)>0$.  By comparison principle for parabolic equations,
we have
$$
u^*(t,x;g_n,y_n,L)\ge \tilde u^{L}(t,x)\quad \forall\,\, n\gg 1.
$$
This implies that
$$
u^*(0,x_n;g_n,y_n,L)\not \to 0
$$
as $n\to\infty$, which is a contradiction. Hence
$$\inf_{y\ge y^*, L/4\le x\le 3L/4, g\in H(f)}u^*(0,x;g,y,L)>0.$$
\end{proof}

\begin{lemma}
\label{unbounded-domain-lm3}
Assume (H1)-(H5). Let $u_0\equiv M(\gg 1)$. Then
$u(t,\cdot;u_0,g\cdot(-t))$ decreases as $t$ increases. Let $u^*(g)(x)=\lim_{t\to\infty} u(t,x;u_0,g\cdot(-t))$ for
$x\in  [0,\infty)$.
Then $u(t,\cdot;u^*(g),g)=u^*(g\cdot t)(\cdot)$ and $\inf_{x\in\RR^+,g\in H(f)}u^*(g)(x)>0$.
\end{lemma}

\begin{proof}
First of all, by comparison principle for parabolic equations, we have $u(t,\cdot;u_0,g)\le u_0$ for any $t>0$ and $g\in H(f)$. Hence
$$u(t+s,\cdot;u_0,g\cdot(-t-s))=u(t,\cdot;u(s,\cdot;u_0,g\cdot(-t-s)),g\cdot(-t))\le u(t,\cdot;u_0,g\cdot(-t))$$
for any $t,s\ge 0$. Therefore, $u(t,\cdot;u_0,g\cdot(-t))$ decreases as $t$ increases.
Let
$$
u^*(g)(x)=\lim_{t\to \infty}u(t,x;u_0,g\cdot(-t))\quad \forall \, x\in [0,\infty).
$$

Next,  for any $g\in H(f)$ and $y\ge y^*$,
$$
u(t,x+y;u_0,g\cdot(-t))\ge u^*(t,x;g\cdot(-t),y,L)\quad \forall\,\, 0<x<L.
$$
It then follows that $\inf_{x> y^*+L/4,g\in H(f)}u^*(g)(x)>0$.
Choose $l>y^*+L/4$ and fix it.  By Proposition \ref{fixed-boundary-prop1},
$$
u(t,x;u_0,g\cdot (-t))\ge u^l(g)(x)\quad {\rm for}\quad 0\le x\le l.
$$
Note that
$$
\inf_{g\in H(f),0\le x\le y^*+L/4}u^l(g)(x)>0.
$$
It then follows that
$$
\inf_{x\ge 0,g\in H(f)} u^*(g)(x)>0.
$$

Now, note that
$$
u(s,x;u_0,g\cdot(-s))\to u^*(g)(x) \ as \ s\to\infty
$$
uniformly in bounded sets. This implies that
\begin{align*}
u(t,x;u^*(g),g)&=\lim_{s\to\infty} u(t,x;u(s,\cdot;u_0,g\cdot(-s)),g)\\
&=\lim_{s\to\infty}u(t+s,x;u_0,g\cdot(-s))\\
&=\lim_{s\to\infty} u(t+s,x;u_0,(g\cdot t)\cdot(-t-s))\\
&=u^*(g\cdot t)(x)
\end{align*}
uniformly in bounded sets. The lemma is thus proved.
\end{proof}





\begin{proof}[Proof of Proposition \ref{unbounded-prop1}]
(1) Let $u^*(g)$ be as in Lemma \ref{unbounded-domain-lm3} for $g\in H(f)$. We prove that $g\mapsto u^*(g)$ satisfies the conclusions in (1).

First, assume that $g_n\to g^*$ as $n\to\infty$. By regularity and
a priori estimates for parabolic equations, there is
$n_k\to \infty$ such that
$$
u^*(g_{n_k})(x)\to u^{**}(x)
$$
uniformly in bounded sets.
We prove that
$u^{**}(x)=u^*(g^*)(x)$. Suppose that $u^{**}(x)\not \equiv u^*(g^*)(x)$. Note that
$u(t,x;u^{**},g^*)$ and $u(t,x;u^*(g^*),g^*)$ exist globally (i.e. exist for all $t\in\RR$)
and
$$\inf u(t,x;u^{**},g^*)>0,\quad \inf u(t,x;u^*(g^*),g^*)>0.
$$
 Therefore,
 $$
 \sup_{t\in\RR} \rho(u(t,\cdot;u^{**},g^*),u(t,\cdot;u^*(g^*),g^*)<\infty
 $$
 and there is $\rho^*>0$ such that
$$
\rho(u^{**}(\cdot),u^*(g)(\cdot))=\rho^*.
$$
Then by Proposition \ref{part-metric-prop}(2),
for any $\tau>0$ there is $\delta>0$ such that
\begin{align*}
\rho^*\leq\rho(u(-n\tau,\cdot;u^{**},g^*),u(-n\tau,\cdot;u^*(g^*),g^*))-n\delta
\ for \ n\in\NN.
\end{align*}
Letting $n\to\infty$, we get a contradiction.
Hence $u^{**}(\cdot)=u^*(g^*)(\cdot)$ and
$u^*(g)(x)$ is continuous in $g$ in open compact topology with respect to $x$.

Next, by Lemma \ref{unbounded-domain-lm3},
 we have that, for any $g\in H(f)$, $u(t,\cdot;u^*(g),g)=u^*(g\cdot t)(\cdot)$.

We prove  now that $u^*(g\cdot t)(x)$ is almost periodic in $t$ uniformly in $x$ in bounded sets.
Note  that for any given $\{\alpha_n^{'}\}\subset\RR$ and $\{\beta_n^{'}\}\subset\RR$,
there are $\{\alpha_n\}\subset\{\alpha_n^{'}\}$ and $\{\beta_n\}\subset\{\beta_n^{'}\}$  such that
$\lim_{n\to\infty}\lim_{m\to\infty}g(t+\alpha_n+\beta_m,x,u)=\lim_{n\to\infty}g(t+\alpha_n+\beta_n,x,u)$ for $(t,x,u)\in\RR^3$.
Assume $\lim_{m\to\infty}g(t+\beta_m,x,u)=g^*(t,x,u)$ and $g^{**}(t,x,u)=\lim_{n\to\infty}g(t+\alpha_n+\beta_n,x,u)$.
It then follows that
$$
\lim_{m\to\infty} u(t+\beta_m,x;u^*(g),g)=u^*(g^*\cdot t)(x)
$$
uniformly in $x$ in bounded sets,
\begin{align*}
\lim_{n\to\infty}\lim_{m\to\infty} u(t+\alpha_n+\beta_m,x;u^*(g),g)&=\lim_{n\to\infty} u(\alpha_n,x;u^*(g^*\cdot t),g^*\cdot t)\\
&=\lim_{n\to\infty}u(t,x;u^*(g^*\cdot\alpha_n),g^*\cdot \alpha_n)\\
&=u^*(g^{**}\cdot  t)(x)
\end{align*}
uniformly in $x$ in bounded sets, and
$$
\lim_{n\to\infty} u(t+\alpha_n+\beta_n,x;u^*(g),g)=u^*(g^{**}\cdot t)(x)
$$
uniformly in $x$ in bounded set. Therefore
$\lim_{n\to\infty}\lim_{m\to\infty} u(t+\alpha_n+\beta_m,x;u^*(g),g)=\lim_{n\to\infty} u(t+\alpha_n+\beta_n,x;u^*(g),g)$.
By regularity and a priori estimates for parabolic equations,
$u(t,x;u^*(g),g)$ is uniformly continuous in $t\in\RR$ and $x\in\RR^+$.
Hence, $u^*(g\cdot t)(x)$ is almost periodic in $t$ uniformly in $x$ in bounded set.

(2) For any $u_0\in\tilde{X}^{++}$ and $g\in H(f)$.
By Proposition \ref{part-metric-prop}(2),
$\rho(u(t,\cdot;u_0,g),u^*(g\cdot t)(\cdot))$ decreases as $t$ increases.
It suffices to prove that
$\lim_{t\to\infty}\rho(u(t,\cdot;u_0,g),u^*(g\cdot t)(\cdot))=0$.
Suppose that this is not true. Then there are $t_n\to\infty$, $g^*\in H(f)$, $u^{**},\tilde u^{**}\in \tilde X^{++}$
with $u^{**}\neq\tilde u^{**}$
 such that $g\cdot t_n\to g^*$,
$u^*(g\cdot t_n)(x)\to u^{**}(x)$ and $u(t_n,\cdot;u_0,g)\to \tilde u^{**}(x)$ locally uniformly in $x\ge 0$.
Note that $u(t,\cdot;u^{**},g^*)$ and $u(t,\cdot;\tilde u^{**},g^*)$ exists for all $t\in\RR$,
$$
\sup_{t\in\RR}\rho(u(t,\cdot;u^{**},g^*),u(t,\cdot;\tilde u^{**},g^*))<\infty,
$$
and there is $\rho^*>0$ such that
$$
\rho(u^{**},\tilde u^{**})=\rho^*.
$$
By the arguments in (1) and Proposition \ref{part-metric-prop}(2), $u^{**}=\tilde u^{**}$, a contradiction. Therefore
$$\lim_{t\to\infty}\rho(u(t,\cdot;u_0,g),u^*(g\cdot t)(\cdot))=0$$ and then
$$
\lim_{t\to\infty} \|u(t,x;u_0,g)-u^*(g)(x)\|_{\tilde X}=0
$$
uniformly in $g\in H(f)$.

(3) Suppose that $\tilde u^{*}(t,x)$ is an entire positive solution of \eqref{unbounded-eq3}, and $\inf_{t\in\RR,x\in\RR^+}\tilde u^{*}(t,x)>0$.
Assume  $\tilde u^*(0,x)\not \equiv u^*(g)(x)$. By the arguments in (1) and Proposition \ref{part-metric-prop}(2), there is
$\delta_1>0$ such that
$$
\rho(\tilde u^*(0,\cdot),u^*(g)(\cdot))\le \rho(\tilde u^*(-n,\cdot),u^*(g\cdot(-n))(\cdot))-n\delta_1
$$
for $n\ge 1$. Letting $n\to\infty$, we get a contradiction. Therefore $\tilde u^*(0,x)\equiv u^*(g)(x)$ and then $\tilde u^*(t,x)\equiv
u^*(g\cdot t)(x)$.

(4) It follows from the fact that, if $f(t,x,u)\equiv f(t,u)$, then  for $u_0\equiv M$, $u(t,x;u_0,g)$ is independent of $x$.
\end{proof}

\section{Spreading-Vanishing Dichotomy in Diffusive KPP Equations with a Free Boundary and Proof of Theorem \ref{main-thm2}}

In this section, we study the spreading/vanishing scenario of \eqref{main-eq} and prove Theorem \ref{main-thm2}.
Throughout this section, we assume (H1)-(H5).

We first prove a lemma. For any given $h_0>0$ and  $u_0$ satisfying \eqref{initial-value},
recall that $(u(t,x;u_0,h_0)$, $h(t;u_0,h_0))$ is the solution of \eqref{main-eq},
and  $x=h(t;u_0,h_0)$ is increasing, and therefore
there exists $h_{\infty}\in(0,+\infty]$ such that
$\lim_{t\to+\infty}h(t;u_0,h_0)=h_{\infty}$.
 To stress the dependence of $h(t;u_0,h_0)$ on $\mu$, we now
write $h_{\mu}(t;u_0,h_0)$ instead of $h(t;u_0,h_0)$ and $h_\infty(\mu)$
instead of $h_\infty$ in the following. If no confusion occurs, we write $h_\mu(t;u_0,h_0)$ as $h_\mu(t)$.

\begin{lemma}
\label{monotone-increase}
For any $t\in(0,+\infty)$, $h_{\mu}(t)$ is a strictly
 increasing function of $\mu$.
\end{lemma}

\begin{proof}
Suppose $0<\mu_{1}<\mu_{2}$. Let $(u_{1},h_{\mu_{1}})$
and $(u_{2},h_{\mu_{2}})$ are the solutions of the following
free boundary problems
\begin{equation*}
\label{free-eq1}
\begin{cases}
u_t=u_{xx}+uf(t,x,u),\quad &t>0, 0<x<h_{\mu_{1}}(t)\cr
h_{\mu_{1}}^{'}(t)=-\mu_{1} u_x(t,h_{\mu_{1}}(t)),\quad &t>0 \cr
u_x(t,0)=u(t,h_{\mu_{1}}(t))=0,\quad &t>0 \cr
h_{\mu_{1}}(0)=h_{0},u(0,x)=u_{0}(x),\quad &0<x\leq h_{0}
\end{cases}
\end{equation*}
and
\begin{equation*}
\label{free-eq2}
\begin{cases}
u_t=u_{xx}+uf(t,x,u),\quad &t>0, 0<x<h_{\mu_{2}}(t)\cr
h_{\mu_{2}}^{'}(t)=-\mu_{2} u_x(t,h_{\mu_{2}}(t)),\quad &t>0 \cr
u_x(t,0)=u(t,h_{\mu_{2}}(t))=0,\quad &t>0 \cr
h_{\mu_{2}}(0)=h_{0},u(0,x)=u_{0}(x),\quad &0<x\leq h_{0}.
\end{cases}
\end{equation*}
Since $0<\mu_{1}<\mu_{2}$, then we have
$$h_{\mu_{1}}^{'}(t)=-\mu_{1} u_x(t,h_{\mu_{1}}(t))
<-\mu_{2} u_x(t,h_{\mu_{1}}(t)).$$
By Proposition \ref{comparison-principle},
we can obtain
$$h_{\mu_{1}}(t)\leq h_{\mu_{2}}(t)  \ \ for \
t\in[0,+\infty).$$

Now we prove $h_{\mu}(t)$ is strictly increasing.
If not, then we can find a first $t^{*}>0$ such
that $h_{\mu_{1}}(t)<h_{\mu_{2}}(t)$ for
$t\in(0,t^{*})$ and $h_{\mu_{1}}(t^{*})=
h_{\mu_{2}}(t^{*})$. It follows that
$$h'_{\mu_{1}}(t^{*})\geq h'_{\mu_{2}}(t^{*}).$$
Compare $u_{1}$ and $u_{2}$ over the region
$$\Omega_{t^{*}}:=\{(t,x)\in\RR: 0<t\leq t^{*},
0\leq x <h_{\mu_{1}}(t)\}.$$
The Strong maximum principle yields $u_{1}(t,x)<
u_{2}(t,x)$ in $\Omega_{t^{*}}$. Hence $w(t,x):=
u_{2}(t,x)-u_{1}(t,x)>0$ in $\Omega_{t^{*}}$ with
$w(t^{*},h_{\mu_{1}}(t^{*}))=0$. It follows that
$w_{x}(t^{*},h_{\mu_{1}}(t^{*}))< 0$, from which
we deduce, in view of
$(u_{1})_{x}(t^{*},h_{\mu_{1}}(t^{*}))<0$ and
$\mu_{1}<\mu_{2}$, that
$$-\mu_{1}(u_{1})_{x}(t^{*},h_{\mu_{1}}(t^{*}))<
-\mu_{2}(u_{2})_{x}(t^{*},h_{\mu_{2}}(t^{*})).
$$
Thus $h'_{\mu_{1}}(t^{*})<h'_{\mu_{2}}(t^{*})$.
But this is a contradiction, which proved our conclusion
that $h_{\mu_{1}}(t)<h_{\mu_{2}}(t)$ for all $t>0$.
\end{proof}

\begin{remark}
If we consider \eqref{main-doub-eq}, for any $t\in(0,+\infty)$,
by Proposition \ref{comparison-principle1} and using the same
argument as Lemma \ref{monotone-increase} we have $g_\mu(t)$
is a strictly monotone decreasing function of $\mu$.
\end{remark}

We now prove Theorem \ref{main-thm2}.

\begin{proof}[Proof of Theorem \ref{main-thm2}]

(1)(i) Suppose that $h_\infty<\infty$. First of all, we claim that
 $h^{'}(t;u_0,h_0)\to 0$ as $t\to\infty$. Assume that the claim is not true. Then there is $t_n\to \infty$ ($t_n\ge 1$) such that
 $\lim_{n\to\infty}h^{'}(t_n;u_0,h_0)>0$. Let $h_n(t)=h(t+t_n;u_0,h_0)$ for $t\ge -1$.
  Note that $h_n(t)\to h_\infty$ as $n\to\infty$ uniformly in $t\ge -1$.
  By \cite[Theorem 2.1]{DuLi},  $\{h_n^{'}(t)\}$ is uniformly bounded and equicontinuous on $[-1,\infty)$.  We then may assume that
  there is a continuous function $h^*(t)$ such that $h_n^{'}(t)\to h^*(t)$ as $n\to\infty$ uniformly in $t$ in bounded sets of $[-1,\infty)$.
  It then follows that $h^*(t)=\frac{d h_\infty}{dt}\equiv 0$ and then $\lim_{n\to\infty}h^{'}(t_n;u_0,h_0)=0$, which is a contradiction.
  Hence the claim holds.

 By regularity and a priori estimates for parabolic equations,
 for any sequence $t_n\to\infty$, there are $t_{n_k}\to\infty$ and
$u^*\in C^1(\RR\times [0,h_\infty])$ and $g^*\in H(f)$ such that
$$
f\cdot t_{n_k}\to g^*
$$
and
$$
\|u(t+t_{n_k},\cdot;u_0,h_0)-u^*(t,\cdot)\|_{C^1([0,h(t+t_{n_k})])}\to 0
$$
as $t_{n_k}\to\infty$. Moreover, we have that $u^*(t,x)$ is an entire solution
of
\begin{equation}
\label{limit-eq1}
\begin{cases}
u_t=u_{xx}+u g^*(t,x,u),\quad 0<x<h_\infty\cr
u_x(t,0)=u(t,h_\infty)=0.
\end{cases}
\end{equation}

Next, we show that $h_{\infty}<\infty$
implies $h_{\infty}\leq l^*$.
Assume that $h_{\infty}\in(l^*,\infty)$. Then
 there exists $\tilde{T}>0$ such that $h(t)>h_{\infty}-\epsilon>l^*$
for all $t\geq\tilde{T}$ and some small $\epsilon>0$.
Consider
\begin{equation}
\label{limit-eq2}
\begin{cases}
v_t=v_{xx}+vf(t,x,v),\quad  0< x< h_{\infty}-\epsilon \cr
v_{x}(t,0)=v(t,h_{\infty}-\epsilon)=0.
\end{cases}
\end{equation}
By comparison principle for parabolic equations,
$$
u(t+\tilde T,\cdot;u_0,h_0)\ge v(t+\tilde{T},\cdot;u(\tilde T,\cdot;u_0,h_0),\tilde T)\quad {\rm for}\quad t\ge 0,
$$
where $v(t+\tilde{T},\cdot;u(\tilde T,\cdot;u_0,h_0),\tilde T)$ is the solution of \eqref{limit-eq2}
with $u(\tilde T,\cdot;u(\tilde T,\cdot;u_0,h_0),\tilde T)=u(\tilde T,\cdot;u_0,h_0)$.
By Proposition \ref{fixed-boundary-prop1}, \eqref{limit-eq2} has a unique time almost periodic positive solution $v_{h_{\infty}-\epsilon}(t,x)$.
Moreover, for any $v_0\ge 0$ and $v_0\not\equiv 0$,
\begin{equation}
\label{limit-eq3}
\|v(t+\tilde T,\cdot;v_0,\tilde T)-v_{h_{\infty-\epsilon}}(t+\tilde T,\cdot)\|\to 0
\end{equation}
as $t\to\infty$.  By \eqref{limit-eq3} and
comparison principle for parabolic equations,
$$
u^*(t,x)>0\quad \forall\,\, t\in\RR,\,\, x\in (0,h_\infty).
$$
It then follows that $u^*_x(t,h_\infty)<0$. This implies that
$$
\limsup_{t\to\infty}u_x(t,h(t);u_0,h_0)<0
$$
and then
$$
\liminf h^{'}(t)=\liminf_{t\to\infty} -\mu u_x(t,h(t);u_0,h_0)>0,
$$
which is a contradiction. Therefore $h_\infty\le l^*$.

We  now  show that $\lim_{t\to\infty}\|u(t,\cdot;u_0,h_0)\|
_{C([0,h(t)])}=0$. Let $\bar{u}(t,x)$
denote the solution of the problem
\begin{equation*}
\label{fixed-boundary-eq14}
\begin{cases}
\bar{u}_t=\bar{u}_{xx}+\bar{u}f(t,x,\bar{u}),\quad &t>0, 0<x<h_{\infty} \cr
\bar{u}_{x}(t,0)=\bar{u}(t,h_{\infty})=0,\quad &t>0 \cr
\bar{u}(0,x)=\tilde{u}_{0}(x), \quad &0\leq x\leq h_{\infty}
\end{cases}
\end{equation*}
where
$$\tilde{u}_{0}(x)=
\begin{cases}
u_{0}(x) & \ for \ 0\leq x\leq h_{0} \\
0 & \ for \ x> h_{0}
\end{cases}
$$
The comparison principle implies that
$$0\leq u(t,x;u_0,h_0)\leq\bar{u}(t,x) \  \ for \ t>0 \ and \
x\in[0,h(t)]$$

If $h_{\infty}<l^*$, then  $\lambda(a,h_{\infty})<0$
and  by Proposition \ref{fixed-boundary-prop1},
 $\bar{u}\to0$ uniformly for $x\in[0,h_{\infty}]$ as
$t\to\infty$. Hence, $\lim_{t\to\infty}\|u(t,\cdot;u_0,h_0)\|
_{C([0,h(t)])}=0$.

If $h_\infty=l^*$, assume that  $\lim_{t\to\infty}\|u(t,\cdot;u_0,h_0)\|
_{C([0,h(t)])}\not =0$. Then there are $t_n\to\infty$ and $u^*\not =0$, $g^*\in H(f)$ such that
$\|u(t_n,\cdot;u_0,h_0)-u^*(\cdot)\|_{C([0,h(t_n)])}\to0$ and $f\cdot t_n\to g^*$ as $t_n\to\infty$. We have $u(t,\cdot;u^*,g^*)$ is an entire solution of
$$
\begin{cases}
u_t=u_{xx}+ug^*(t,x,u),\quad 0<x<l^*\cr
u_x(t,0)=u(t,l^*)=0.
\end{cases}
$$
By Hopf lemma  for parabolic equations, we have  $u_x(t,l^*;u^*,g^*)<0$. This implies that
$$
\lim_{n\to\infty} h^{'}(t_n)=-\lim_{n\to\infty}\mu u_x(t_n,h(t_n);u_0,h_0)>0,
$$
which is a contradiction again.

(1)(ii)
First note that for any fixed $x$, $u^l(g)(x)$ is increasing in $l$ and $u^l(g)(x)\le u^*(g)(x)$.
Then there is $\tilde u^*(g)(x)$ such that
$$
\lim_{l\to\infty}u^l(g)(x)=\tilde u^*(g)(x)\le u^*(g)(x)
$$
locally uniformly in $x$.

We claim that
$$
\tilde u^*(g)(x)\equiv u^*(g)(x).
$$
In fact, by Lemma \ref{unbounded-domain-lm2},
$$
\inf_{x\ge 0,g\in H(f)}\tilde u^*(g)(x)>0.
$$
Note that $u(t,x;\tilde u^*(g),g)=\tilde u^*(g\cdot t)(x)$. Then by Proposition \ref{unbounded-prop1},
  $u^*(g)(x)\equiv \tilde u^*(g)(x)$.

Note that for any $T>0$ satisfying $h(T)>l^*$,
$$
u(t+T,x;u_0,h_0)\ge u^l(t,x;u(T,\cdot;u_0,h_0),f\cdot T)\quad \forall\,\, t\ge 0,
$$
where $u^l(t,x;u(T,\cdot;u_0,h_0),f\cdot T)$ is the solution of \eqref{fixed-boundary-eq3} with
$g=f\cdot T$, $l=h(T;u_0,h_0)$, and $u^l(0,x;u(T,\cdot;u_0,h_0),f\cdot T)=u(T,x;u_0,h_0)$.
Note also that
$$
u^l(t,x;u(T,\cdot;u_0,h_0),f\cdot T)-u^l(f\cdot (t+T))(x)\to 0
$$
as $t\to\infty$ uniformly in $x\in [0,l]$ and
$$
u^l(f\cdot (t+T))(x)-u^*(f\cdot(t+T))(x)\to 0
$$
as $l\to\infty$ locally uniformly in $x\in [0,\infty)$. It then follows that
$$
u(t,x;u_0,h_0)-u^*(f\cdot t)(x)\to 0
$$
as $t\to\infty$ locally uniformly in $x\in [0,\infty)$.

(2) If $h_0\ge l^*$, then $h_\infty> h_0\ge l^*$. (2) then follows from (1).

(3) Assume that $h_{0}<l^*$.
Let
$$
\mu^{*}=\sup\{\mu\,|\, h_\infty(\mu)<\infty\}.
$$

We claim that $\mu^{*}\in\{\mu|h_\infty(\mu)<\infty\}$
when $\{\mu\,|\,h_\infty(\mu)<\infty\}\not =\emptyset$. Otherwise
$h_\infty(\mu^{*})=\infty$. It means that we can find $T>0$
such that $h_{\mu^{*}}(T)>l^*$. By the continuous dependence of
$h_\mu$ on $\mu$, there is $\epsilon>0$ small such that
$h_\mu(T)>l^*$ for all $\mu\in[\mu^{*}-\epsilon,\mu^{*}+\epsilon]$.
Hence, for all such $\mu$ we have
$$h_\infty(\mu)=\lim_{t\to\infty}h_\mu(t)>h_\mu(T)>l^*$$
Thus, $h_\infty(\mu)=\infty$. This implies that
$[\mu^{*}-\epsilon,\mu^{*}+\epsilon]\cap\{\mu|h_\infty(\mu)<\infty\}=\emptyset$,
and it is a contradiction to the definition of $\mu^{*}$. So we proved the claim
that $\mu^{*}\in\{\mu|h_\infty(\mu)<\infty\}$.

For $\mu>\mu^{*}$, we get $h_\infty(\mu)=\infty$. If not,
it must have $\mu\leq\mu^{*}$, and it is a contradiction.
Then spreading happens.

For $\mu\leq\mu^{*}$, by the Lemma \ref{monotone-increase}
we can obtain
$$h_\mu(t)\leq h_{\mu^{*}}(t) \ for \ all \ t\in(0,+\infty)$$
It follow that $h_\infty(\mu)\leq h_\infty(\mu^{*})<\infty$,
and vanishing happens.
\end{proof}

\section{Remarks}
We have examined the dynamical behavior of the population $u(t,x)$
with spreading front $x=h(t)$ determined by \eqref{main-eq}, and
proved that for this problem a spreading-vanishing dichotomy holds
(see Theorem \ref{main-thm2}). In this section, we discuss how  the techniques for \eqref{main-eq} can be modified to study the
 double fronts free boundary \eqref{main-doub-eq}.

First, note that the existence and uniqueness results for solutions of \eqref{main-eq} with given initial datum $(u_0,h_0)$
 can be
extended to \eqref{main-doub-eq} using the same arguments as in Section 5
\cite{DuLi}, except that we need to modify the transformation in the
proof of Theorem 2.1 in \cite{DuLi} such that both boundaries are
straightened. In particular, for given
$g_0<h_0$ and $u_0$ satisfying \eqref{initial-value-1},
the system \eqref{main-doub-eq} has a unique global solution $(u(t,x;u_0,h_0,g_0),h(t;u_0,h_0,g_0),g(t;u_0,h_0,g_0))$ with
$u(0,x;u_0,h_0,g_0)=u_0(x)$, $h(0;u_0,h_0,g_0)=h_0$, $g(0;u_0,h_0,g_0)=g_0$.
Moreover, $g(t)$ decreases and $h(t)$ increases as $t$ increases. Let
$g_\infty=\lim_{t\to\infty}g(t;u_0,h_0,g_0)$ and $h_\infty=\lim_{t\to\infty}h(t;u_0,h_0,g_0)$.

We next consider the spreading-vanishing dichotomy for \eqref{main-doub-eq}. To this end,
We assume $(H4)^*$ instead of $(H4)$,

\medskip
\noindent {\bf $(H4)^*$} {\it There is $L^*\ge 0$ such that
 $\inf_{y\in\RR, \, l\ge L^*} \tilde\lambda(a(\cdot,\cdot+y),l)>0$.}
\medskip

\noindent  Consider the following  fixed boundary problem on $\RR^1$,
\begin{equation}
\label{glob-eq}
u_t=u_{xx}+uf(t,x,u) \quad x\in(-\infty,\infty).
\end{equation}
For given $u_0\in C^b_{\rm unif}(\RR,\RR^+)$, let $u(t,x;u_0)$ be the solution of
\eqref{glob-eq} with $u(0,x;u_0)=u_0(x)$.
By the similar arguments as those in Theorem \ref{main-thm1}, we can prove

\begin{proposition}
\label{positive-almost-periodic-solution-doub-prop}
Assume (H1), (H2), $ (H4)^*$, and (H5).
 \eqref{glob-eq} has a unique time almost periodic positive solution
$u^*(t,x)$ and for any $u_0\in C^b_{\rm unif}(\RR,\RR^+)$ with
$\inf_{x\in(-\infty,\infty)}u_0(x)>0$,
$\lim_{t\to\infty}\|u(t,\cdot;u_0)-u^*(t,\cdot)\|_{C(\RR)}=0$.
\end{proposition}

We now have the following spreading-vanishing dichotomy for \eqref{main-doub-eq}.

\begin{proposition}
\label{spreading-vanishing-doub-prop}
Assume (H1), (H2), $(H4)^*$, and (H5). For given $h_0>0$ and  $u_0$ satisfying \eqref{initial-value-1},
the following hold.
\begin{itemize}
\item[(1)]
Either

(i) $h_\infty-g_\infty\le L^*$ and
$\lim_{t\to+\infty} u(t,x;u_0,h_0,g_0)=0$ uniformly in $x$

or

(ii)  $h_{\infty}=-g_{\infty}=\infty$ and
$\lim_{t\to\infty}[u(t,x;u_0,h_0,g_0)-u^*(t,x)]=0$
locally uniformly for $x\in(-\infty,\infty)$, where $u^*(t,x)$ is the
unique time almost periodic positive solution of \eqref{glob-eq}.

\item[(2)] If $h_0-g_0\ge L^*$, then $h_\infty=-g_{\infty}=\infty$.

\item[(3)]
Suppose $h_0-g_0<L^*$. Then there exists $\mu^{*}>0$
such that spreading occurs if $\mu>\mu^{*}$ and vanishing occurs if $\mu\le \mu^{*}$.
\end{itemize}
\end{proposition}

\begin{proof}
(1) Observe that we have either $h_\infty-g_\infty<\infty$ or $h_\infty-g_\infty=\infty$.

Suppose that $h_\infty-g_\infty<\infty$. By $(H4)^*$ and the similar arguments as those in Theorem \ref{main-thm2}(1)(i),
we must have $h_\infty-g_\infty\le L^*$ and $u(t,x;u_0,h_0,g_0)\to 0$ as $t\to\infty$.

Suppose that $h_\infty-g_\infty=\infty$. We first claim that $h_\infty=-g_\infty=\infty$.
In fact, if the claim does not hold, without loss of generality, we may assume that $-\infty<g_\infty<h_\infty=\infty$.
By the similar arguments as those in Theorem \ref{main-thm2}(1)(i), we have $g^{'}(t)\to 0$ as $t\to\infty$.
Let $T^*>0$ be such that $h(T^*)-g(T^*)>L^*$. Then by $(H4)^*$,
$$
\inf_{t>T^*,x\in [g(T^*),h(T^*)]}u(t,x;u_0,h_0,g_0)>0.
$$
Let $t_n\to\infty$ be such that $f(t+t_n,x,u)\to g^*(t,x,u)$ and $u(t+t_n,x;u_0,h_0,g_0)\to u^*(t,x)$.
Then $u^*(t,x)$ is the solution of
$$
\begin{cases}
u_t=u_{xx}+u g^*(t,x,u),\quad g_\infty<x<\infty\cr
u(t,g_\infty)=0,
\end{cases}
$$
and
$\inf_{t\in\RR,x\in [g(T^*),h(T^*)]}u^*(t,x)>0$. Then by Hopf Lemma for parabolic equations,
$$
u^*_x(t,g_\infty)>0.
$$
This implies that
$$g^{'}(t+t_n;u_0,h_0,g_0)\to -\mu u^*_x(t,g_\infty)<0,
$$
which contradicts to the fact that $g^{'}(t;u_0,h_0,g_0)\to 0$ as $t\to\infty$. Hence $(g_\infty,h_\infty)=(-\infty,\infty)$.
By the similar arguments as those in Theorem \ref{main-thm2}(1)(ii), we have
$$
\lim_{t\to\infty}[u(t,x;u_0,h_0,g_0)-u^*(t,x)]=0
$$
locally uniformly in $x$ in bounded sets.

(2) and (3) follows from the similar arguments as those in Theorem \ref{main-thm2} (2) and (3), respectively.
\end{proof}

\section*{Acknowledgements}

Fang Li would like to thank the China Scholarship Council for
financial support during the two years of her overseas study and to
express her gratitude to the  Department of Mathematics and
Statistics, Auburn University  for its kind hospitality.

\end{document}